\documentclass[a4paper,reqno]{amsart}
\usepackage{mathtools} 
\usepackage{amstext, amsthm, amssymb, amsfonts}
\usepackage{booktabs}
\usepackage{graphicx, fancyvrb}
\usepackage{bm}
\usepackage[T1]{fontenc}
\graphicspath{{figures/}}
\usepackage{enumitem}
\usepackage{mathrsfs}
\usepackage[dvipsnames]{xcolor}
\usepackage[margin=2.5cm]{geometry}
\usepackage[normalem]{ulem}
\usepackage{dsfont}  

\usepackage{comment}

\usepackage{hyperref} 
\hypersetup{colorlinks}

\theoremstyle{plain}

\newtheorem{definition}{Definition}[section]
\newtheorem{theorem}[definition]{Theorem}
\newtheorem{lemma}[definition]{Lemma}

\newtheorem{prop}[definition]{Proposition}
\newtheorem{assumption}[definition]{Assumption}

\theoremstyle{definition}
\newtheorem{remark}[definition]{Remark}

\newcommand{\N}{\mathbb{N}}
\newcommand{\one}{\mathds{1}}
\newcommand{\half}{\frac{1}{2}}
\newcommand{\inner}[3][]{\langle #2 , #3 \rangle_{#1}}

\newcommand{\innerB}[3][]{\Big\langle #2 , #3 \Big\rangle_{#1}}
\newcommand{\E}{\mathbb{E}}
\newcommand{\C}{\mathbb{C}}

\newcommand{\tL}{L}
\newcommand{\tK}{\widehat K}

\newcommand{\nuplus}{\max(0,\nu)}
\newcommand{\R}{\mathbb{R}}
\newcommand{\tone}{t + h \xi(\omega)}
\newcommand{\ttwo}{t + h (1-\xi(\omega))}
\renewcommand{\P}{\mathbb{P}}
\renewcommand{\Re}{\mathrm{Re}}
\newcommand{\F}{\mathcal{F}}
\newcommand{\diff}[1]{\,\mathrm{d}#1}
\newcommand{\verttt}{{\vert\kern-0.25ex\vert\kern-0.25ex\vert}}
\newcommand{\initvec}{v}
\newcommand{\genvec}{v}

\newcommand{\randvec}{\eta}

\newcommand{\kone}{k^1}
\newcommand{\ktwo}{k^2}
\newcommand{\kones}{k^1_h}
\newcommand{\ktwos}{k^2_h}
\newcommand{\koness}{k^1_{hh}}
\newcommand{\ktwoss}{k^2_{hh}}

\newcommand{\ellone}{\ell^1}
\newcommand{\elltwo}{\ell^2}
\newcommand{\ellones}{\ell^1_h}
\newcommand{\elltwos}{\ell^2_h}

\newcommand{\ellonearg}{\theta^1,\ell^1}
\newcommand{\elltwoarg}{\theta^2,\ell^2}

\newcommand{\vertttb}{{\big\vert\kern-0.25ex\big\vert\kern-0.25ex\big\vert}}
\newcommand{\verttttB}{{\Big\vert\kern-0.25ex\Big\vert\kern-0.25ex\Big\vert}}

\title[Error Bound And Stability Analysis For a Randomized SDIRK method]{Error
Bound and Stability Analysis for a Randomized Singly Diagonally Implicit
Runge--Kutta Method}
\author[M.~Eisenmann]{Monika Eisenmann}
\author[M.~Jans]{Marvin Jans}
\author[R.~Kruse]{Raphael Kruse}
\author[H.~Podhaisky]{Helmut Podhaisky}
\email{marvin.jans@math.lth.se}
\subjclass[2020]{Primary  65L20, 65C10, 65L05, 65L06}
\date{\today}

\thanks{The first and the second author were supported in part by the Swedish
  Research Council under the grant 2023-03930, eSSENCE: The e-Science
  Collaboration and the Crafoord foundation.
}

\begin{document}

\allowdisplaybreaks

\begin{abstract}
A randomized Singly Diagonally Implicit Runge--Kutta (SDIRK) method,
based on the randomized trapezoidal rule as the underlying quadrature
scheme, is proposed.
Every realization of the scheme is an algebraically
stable SDIRK method of at least second order.
The main result is the proof that the randomized scheme converges with order
$2.5$ in the root mean square sense under low regularity assumptions.
Numerical experiments illustrate the robustness of the new scheme
when applied to nonsmooth problems.
\end{abstract}

\maketitle

\section{Introduction}

Ordinary differential equations (ODEs) arise in a wide range of scientific and
engineering models.
In the following, we consider a general initial value problem of the type
\begin{align*}
  u'(t)=f(t,u(t)) \in \R^d, \quad t \in (0,T), \quad u(0) = u_0 \in \R^d,
\end{align*}
for $d \in \N$ and $T \in (0,\infty)$ that can represent a large amount of
concrete examples.
While this type of equation is used in many applications, it is typically not
possible to find an analytical solution.
Thus, a numerical approximation needs to be found.
We discretize the equation in time, meaning that for $N \in \N$, we define the
grid $t_n = nh$, $n \in \{0,\dots,N\}$, for a step size $h = \frac{T}{N}$ and
find approximations of the exact solution $u(t_n) \approx u^n$ at a grid point.
In this paper, we concentrate on a scheme with good stability properties that is
suited for stiff equations~\cite{Hairer1996}.
Specifically, we employ a singly diagonally implicit $s$-stage Runge--Kutta method (SDIRK)
of the form
\begin{align*}
  u^{n+1} &= u^n + h \sum_{i = 1}^s b_i k_{n+1}^{i}\\
  k_{n+1}^{i} &= f\Big(t_n+c_i h, u^{n} +
  h\sum_{j = 1}^s a_{i,j} k_{n+1}^{j}\Big),
  \quad \text{for} \; i \in  \{1,\ldots,s\}.
\end{align*}
with a Butcher tableau
\begin{align*}
  \begin{array}{c|c}
    c & A   \\ \hline
    & b^\top
  \end{array}, \quad
  A = (a_{i,j})_{i,j \in \{1,\dots,s\}}, \quad
  b = (b_i)_{i\in \{1,\dots,s\}}, \quad \text{and} \quad
  c = (c_i)_{i\in \{1,\dots,s\}},
\end{align*}
where the matrix $A$ is a lower triangular matrix with identical diagonal
elements, so that only $s$ systems of the size $d$ have to be solved in every
step.
In the following, we focus on second-order, two-stage schemes and aim to
increase their order of convergence using randomization of specific entries of
$A$.

Randomization for the approximation of integrals is a well-known strategy
and leads to methods that can remain effective in cases with low regularity
of the integrand where deterministic schemes fail to converge.
The stratified sampling approach, as we use it here, was introduced in
\cite{haber1966,haber1967} for integration problems as randomized quadrature
methods.
In \cite{Kruse2017}, the randomized Riemann sum method and, in \cite{Wu2022},
the randomized trapezoidal quadrature method are further analyzed.

Early work on randomized methods for solving ODEs includes \cite{Dodson,
stengle1990, stengle1995}.
Additional randomized time stepping methods similar to ours can be found, for
example, in \cite{Jentzen2009,Kruse2017} for a randomized forward Euler method,
in \cite{Eisenmann2019} for a randomized backward Euler method, in
\cite{Kruse2017} for a random two-stage explicit Runge--Kutta method, and in
\cite{2025-bochacik_convergence} for a higher order implicit Runge--Kutta
method.
For slightly different problem classes, we also mention \cite{Bochacik.2023,
HuEtAll.2024, KruseWu.2019} and for quasi Monte Carlo Runge--Kutta methods see
\cite{Coulibaly1999,Kainhofer2003,Lecot2001}. Additionally, similar
randomization approaches have been combined with Taylor methods and
exponential integrators in \cite{Bochacik.2023B} and \cite{Hofmanova.2020},
respectively.

In \cite{2025-bochacik_convergence}, the authors
propose an implicit random $\theta$-method and prove convergence of the
method with up to order $1.5$.
While their method is not A-stable for every random realization, they show it is asymptotically stable in a generalized probabilistic sense.
In this paper, by following a similar path, we introduce a randomized
two-stage SDIRK method and perform a rigorous error analysis.
More precisely, we provide error bounds that show a convergence order of up to
order $2.5$ for a coefficient function $f$ that is only two times
differentiable.
We refer to Assumption~\ref{ass:f_high_conv} for the exact regularity requirements.
We carefully derive the error constants and make sure they do not exponentially
grow with factors depending on the (local) Lipschitz constant of the function
$f$.
Additionally, we show that our method is algebraically stable for every random
choice. This implies, in particular, that the method is A-stable and there is no
step size restriction depending on the Lipschitz constant.
Together, this indicates that the proposed method is applicable to stiff
problems, e.g., arising from the spatial discretization of parabolic PDEs.

Thus, compared to the Runge--Kutta method studied in \cite{2025-bochacik_convergence},
our method has improved properties in two different ways.
First, our method achieves a higher convergence order with error constants that
do not grow exponentially with the Lipschitz constant of $f$.
Second, our method is A-stable for every realization of the random parameter.

Randomized methods with arbitrary convergence orders have also been studied
in~\cite{Daun2011,Heinrich2008, Kacewicz2004} in the framework of
information-based complexity. In~\cite{Daun2011,Heinrich2008, Kacewicz2004},
they assume that $f$ is bounded, which can be quite restrictive. Furthermore, for
the proof of convergence higher than $1.5$, the same regularity assumptions in
both input parameters of the function $f$ are required. One aim of this paper is
to ease the regularity on the temporal argument. Additionally, we can lift the
assumptions on the boundedness of $f$ and its derivatives. This shows that when
the ODE simplifies to an integration problem, we get a similar convergence
result as in~\cite{Wu2022} for the randomized trapezoidal quadrature method.
This is desirable, since the temporal argument can be more irregular than the
spatial argument.

For the spatial argument, we require a stronger regularity assumption compared to
the temporal argument.
Because of this, it is not sufficient to look only at autonomous ODEs in the
error analysis.
Despite the higher regularity assumption on the spatial argument, the
assumptions are not necessarily stronger than the H\"older continuity assumption
on the spatial argument used in~\cite{Daun2011,Heinrich2008,Kacewicz2004}.

The remainder of the paper is organized as follows:
Section~\ref{sec:setting} establishes the regularity requirements for the
underlying problem.
Section~\ref{sec:method} introduces the randomized implicit Runge--Kutta method
along with the classical conditions for order 2.
Section~\ref{sec:existence} provides proofs for the existence and measurability
of the numerical solution.
It is shown in Section~\ref{sec:stability} that every realization of the random
Runge--Kutta method is algebraically stable and, therefore, A-stable.
Section~\ref{sec:taylor} derives the Taylor expansion of the local error.
The main result, the convergence of order 2.5, is established in
Section~\ref{sec:convergence}.
Finally, Section~\ref{sec:numExp} validates and illustrates the theoretical
findings through numerical experiments.
Appendix~\ref{sec:appendix} establishes an auxiliary integral identity used in
the error analysis.

\section{Setting}
\label{sec:setting}
Throughout this paper, we consider an ordinary differential equation of the form
\begin{equation}
  \label{eq:problem}
  \begin{cases}
  u'(t)=f(t,u(t)), &t\in(0,T);\\
  u(0)=u_0.
  \end{cases}
\end{equation}
Hereby, the function $f \colon [0,T] \times \R^d \to \R^d$ and the initial value
$u_0 \in \R^d$ are given.

In the following, we introduce two sets of assumptions on the function $f$.
The first set is sufficient to ensure existence and boundedness of the exact
solution.
Moreover, in this setting we can already show some basic properties for our
numerical approximation. These include the existence of a numerical solution,
the measurability in the random parameter introduced by the method, the
boundedness of the numerical approximation, and its stability.

By $\langle \cdot, \cdot \rangle$ and $\| \cdot \|$ we denote the standard
Euclidean inner product and the Euclidean norm on $\R^d$, respectively.
To keep the notation simple, we also denote the matrix norm on $\R^{d,d}$ and
the tensor norm on $\R^{d,d,d}$ that are induced by the Euclidean norm in the
same fashion.

\begin{assumption}
  \label{ass:f_low}
  The function $f \colon [0,T] \times \R^d \to \R^d$ satisfies the
  following conditions:
  \begin{enumerate}[label={(\alph*)}, ref={\ref{ass:f_low}~(\alph*)}]
    \item\label{ass:f_caratheodory} The function $f$ fulfills a Carath\'{e}odory
      condition, i.e., for every fixed $v \in \R^d$ the function $\R \ni t
      \mapsto f(t, v) \in \R^d$ is measurable and for almost every fixed $t \in
      [0,T]$ the function $\R^d \ni v \mapsto f(t, v) \in \R^d$ is continuous.

    \item \label{ass:f_one_lip_B} The function $f$ fulfills a one-sided
      Lipschitz condition, i.e., there exists a constant $\nu \in \R$ with
      \begin{equation}
        \label{ieq:f_one_lip_B1}
        \inner{f(t,v)-f(t,w)}{v-w} \leq \nu \|v-w\|^2
      \end{equation}
      is fulfilled for all $t \in [0, T]$ and all $v,w \in \mathbb{R}^d$.

    \item \label{ass:f_bound} For every $\kappa >0$, there exists a constant
      $M_{\kappa}>0$ such that
      \begin{equation*}
        \|f(t,v)\| \leq M_{\kappa}
      \end{equation*}
      for all $t \in  [0, T]$ and all $v\in\R^d$ with $\|v\| \leq \kappa$.
  \end{enumerate}
\end{assumption}

The function $f$ is called \emph{contractive} if the inequality
\eqref{ieq:f_one_lip_B1} holds with $\nu=0$.

Assumption~\ref{ass:f_low} is, in general, not sufficient to ensure the existence
of a solution in the classical sense.
Instead, we employ the weaker notion of a solution in the Carath\'{e}odory sense,
i.e., the goal is to find an absolutely continuous function $u \colon [0,T] \to
\R^d$ which satisfies the differential equations \eqref{eq:problem} for almost
every $t \in (0,T)$.
We refer to \cite[Section~I.5]{Hale1980} for more details on this solution
concept.

\begin{lemma}\label{lem:uexist}
  Let Assumption~\ref{ass:f_low} be fulfilled. Then there exists a unique
  solution $u$ to \eqref{eq:problem} in the Carath\'{e}odory sense.
  Moreover, there exists a constant $M_{\mathrm{exact}}$ such that
  \begin{equation}\label{eq:Mexact}
    \| u(t) \|
    \leq \mathrm{e}^{t \nu } \|u_0\|
    + \int_{0}^{t} \mathrm{e}^{(t-\tau) \nu } \|f(\tau,0)\| \diff{\tau}
    \leq \mathrm{e}^{T \max(\nu,0)} \|u_0\|
      + \int_{0}^{T} \mathrm{e}^{(T-\tau) \max(\nu,0)} \|f(\tau,0)\| \diff{\tau}
    =: M_{\mathrm{exact}}
  \end{equation}
  for all  $t\in [0,T]$.
\end{lemma}

\begin{proof}
  Existence and uniqueness of a local solution $u$ follows from an
  application of \cite[Theorem~8.1]{Deimling1977}.
  Using \cite[Theorem~112G]{Butcher1987}, we can extend the local solution to
  the global interval $[0,T]$.
  Note that the proof of \cite[Theorem~112G]{Butcher1987} relies on the
  existence result \cite[Theorem~112C]{Butcher1987}, which requires $f$ to be
  continuous.
  To adjust this to our setting, it is sufficient to replace
  \cite[Theorem~112C]{Butcher1987} by \cite[Theorem~8.1]{Deimling1977}.

  It remains to prove the claimed norm bound. Since $u$ is almost everywhere
  differentiable, it follows that $\|u(t)\|\frac{\diff{}}{\diff{t}} \|u(t)\|
  = \frac{1}{2} \frac{\diff{}}{\diff{t}} \|u(t)\|^2 = \inner{u'(t)}{u(t)}$
  for almost every $t\in[0,T]$, see \cite[Lemma~3.2.iv]{Deimling1977}.
  Inserting the differential equation \eqref{eq:problem} then yields
  \begin{align*}
    \|u(t)\|\frac{\diff{}}{\diff{t}} \|u(t) \|
    = \inner{f(t,u(t))}{u(t)} \leq \nu \|u(t)\|^2 + \|f(t,0)\|\|u(t)\|,
    \quad \text{a.e. } t\in [0,T],
  \end{align*}
  where we also made use of Assumption~\ref{ass:f_one_lip_B}.
  For $t \in [0,T]$ such that $u(t) \neq 0$, we divide the inequality by
  $\|u(t)\|$. It then follows that $\frac{\diff{}}{\diff{t}} \|u(t) \|\leq \nu
  \|u(t)\|+\|f(t,0)\|$ for such values of $t$.

  For values $t \in [0,T]$, where $u(t) = 0$ and $u$ is differentiable, we apply
  the fact that the norm is a continuous function, and obtain
  \begin{align*}
    \frac{\diff{}}{\diff{t}}\|u(t)\|
    &= \lim_{h \to 0} \frac{\|u(t+h)\| - \|u(t)\|}{h}
    = \lim_{h \to 0} \frac{\|u(t+h)\|}{h}
    = \Big\| \lim_{h \to 0} \frac{u(t+h) - u(t)}{h} \Big\|
    = \|u'(t)\|.
  \end{align*}
  Thus, for $t \in [0,T]$ where $u$ is differentiable and $u(t) = 0$, we again
  find that $\frac{\diff{}}{\diff{t}} \|u(t)\| = \|f(t,0)\| = \nu \|u(t)\| +
  \|f(t,0)\|$.
  The last step is to apply Gr\"onwall's lemma (compare
  \cite[Lemma~7.3.2]{Emmrich.2004}) to obtain
  \begin{align*}
    \|u(t)\| \leq \mathrm{e}^{\nu t}\|u_0\| +
    \int_0^t \mathrm{e}^{(t-\tau) \nu} \|f(\tau,0)\| \diff{\tau}.
  \end{align*}
  Taking the supremum over $t \in [0,T]$ completes the proof of
  \eqref{eq:Mexact}.
\end{proof}

To prove higher order convergence for the scheme introduced in the next section,
stronger assumptions on $f$ and, specifically, on its partial derivatives are
required.
Before stating these assumptions, we recall the standard multi-index notation
for partial derivatives of $f \colon [0,T] \times \R^d \to \R^d$
as in \cite[Appendix~A]{Evans2010}.
For a multi-index $\bm{\alpha}=(\alpha_0,\dots,\alpha_d)$ with $d+1$ entries, we
consider the order of $\bm{\alpha}$ given by $|\bm{\alpha}| = \sum_{i = 0}^d \alpha_i$.
We then write $f_{\bm{\alpha}} = \partial_0^{\alpha_0} \dots
\partial_{d}^{\alpha_d} f$ for the $|\bm{\alpha}|$-order partial derivative of
$f$.

In the following, there are a few specific partial derivatives that appear
in the analysis often and therefore obtain a shortened notation.
For $i \in \{0,\dots,d\}$, let $e_i$ be the $i$-th unit vector in $\R^{d+1}$.
We then denote the partial derivative with respect to the first argument by
$f_t\coloneqq f_{e_0}$ and the second argument as $f_u:=(f_{e_i})_{i \in \{1,\dots,d\}}$.
Similarly, we define $f_{tt}:=f_{2e_0}$, $f_{tu}:=(f_{e_0+e_i})_{i \in \{1,\dots,d\}}$,
and $f_{uu}:=(f_{e_i + e_j})_{i,j \in \{1,\dots,d\}}$.
Note that $f_u$ and $f_{tu}$ are matrices in $\R^{d,d}$ and $f_{uu}$ is a tensor
in $\R^{d,d,d}$.
We often use the fact that $f_{tu} = f_{ut}$, see
\cite[Theorem~5.2.3.1]{Evans2010}.
If we do not explicitly state that a partial derivative is continuous, it is
understood in the weak sense.

\begin{assumption}\label{ass:f_high_conv}
  The function $f \colon [0,T] \times \R^d \to \R^d$ admits weak partial
  derivatives up to order two and satisfies the following conditions:
  \begin{enumerate}[label={(\alph*)}, ref={\ref{ass:f_high_conv}~(\alph*)}]
    \item \label{ass:f_one_lip_B2}
      The function $f$ fulfills a one-sided Lipschitz condition, i.e., there
      exists a constant $\nu \in \R$ such that for all $t \in [0, T]$ and
      $v,w \in \mathbb{R}^d$
      \begin{equation*}
        \inner{f(t,v)-f(t,w)}{v-w} \leq \nu \|v-w\|^2.
      \end{equation*}

    \item \label{ass:f2_bound}
      For every $\kappa >0$, there exists an $L^2$-integrable function
      $\gamma_{\kappa} \colon [0,T]\to \R\cup\{+\infty\}$ such that for every
      multi-index $\bm{\alpha}$ of order two and every $t\in [0,T]$ and
      $v\in \R^d$ with $\|v\|\leq \kappa$
      \begin{equation}\label{eq:f_bound}
        \|f_{\bm\alpha}(t,v)\|\leq \gamma_{\kappa}(t).
      \end{equation}
    \item \label{ass:f2_hold} There exists $\sigma \in [0,1]$ such
      that for every $\kappa>0$, there exists an $L^2$-integrable function
      $L_{\kappa} \colon [0,T] \to \R\cup\{+ \infty\}$ such that for every
      multi-index $\bm{\alpha}$ of order two
      \begin{equation}\label{eq:f_hold}
        \|f_{\bm\alpha}(t,v)-f_{\bm\alpha}(t,w)\|\leq L_{\kappa}(t) \|v - w\|^{\sigma}
      \end{equation}
      for every $t\in [0,T]$ and $v, w \in\R^d$ with
      $\|v\|, \|w\| \leq \kappa$.
  \end{enumerate}
\end{assumption}

From Assumption~\ref{ass:f_high_conv} and Lemma~\ref{lem:implied_continuity}
it follows that the second derivative of the exact solution $u$ to
\eqref{eq:problem} is continuous and the third derivative is $L^2$-integrable.

\begin{remark}
  To keep the notation simple, we do not distinguish between a mapping and its
  corresponding $L^2$-equivalence class.
\end{remark}

\begin{remark}\label{rem:C}
	We will use a generic constant $C$, which can have different values at
	different places. The only exception is that we have three explicit constants
  $C_1$, $C_2$, and $C_3$ that are stated in Lemma~\ref{lem:alg_bound},
  Lemma~\ref{lem:error_bound}, and Lemma~\ref{lem:B_conv}.
  We state them as they grow exponentially in the error
	constant. Thus, they should not include potentially large constants such as
	the Lipschitz constant of $f$. This helps to ensure that the error constant in
	our error bound remains reasonable.
  It is important to note that all constants are independent of discretization
  parameters introduced later in this paper.
\end{remark}

\begin{lemma} \label{lem:implied_continuity}
  Let Assumption~\ref{ass:f2_bound} be fulfilled. Then the partial
  derivatives $f_t$ and $f_u$ satisfy the following continuity bounds:
  For every $\kappa>0$, there exists an $L^2$-integrable function
  $\tL_{\kappa} \colon [0,T]\to \R\cup\{+\infty\}$
  such that for every $t \in [0,T]$ and $v, w\in \R^d$ with $\|v\|, \|w\|\leq \kappa$
  \begin{align}
    \label{eq:f1_spatial}
    \|f_{t}(t,v)-f_{t}(t,w)\|+\|f_u(t,v)-f_u(t,w)\| \leq \tL_{\kappa}(t)\|v-w\|.
  \end{align}
  Furthermore, there exists a uniform constant $C_{\kappa} > 0$ such that for every
  $s,t\in[0,T]$ and $v \in \R^d$ with $\|v\|\leq \kappa$
  \begin{align}
    \label{eq:f1_temporal}
    \|f_{t}(t,v)-f_{t}(s,v)\|+\|f_u(t,v)-f_u(s,v)\| \leq C_{\kappa} |t-s|^{\half}.
  \end{align}
  Consequently, $f_t$ and $f_u$ admit representatives that are continuous on
  $[0,T] \times \R^d$ and $f$ is locally Lipschitz.
  In particular, for every $\kappa > 0$, there exists a uniform constant
  $C_{\kappa} > 0$ such that
  \begin{align}\label{eq:f0_Lip}
    \|f(t,v)-f(s,w)\| \leq C_{\kappa}(|t-s|+\|v-w\|)
  \end{align}
  for every
  $s,t \in[0,T]$ and $v,w\in\R^d$ with $\|v\|, \|w\|\leq \kappa$.
\end{lemma}

\begin{proof}
  Let $\kappa>0$ be fixed and let $g$ represent either $f_t$ or $f_u$.

  First, we prove the estimate \eqref{eq:f1_spatial}.
  Take $v, w \in \R^d$ with $\|v\|, \|w\| \le \kappa$ arbitrarily.
  Since the ball of radius $\kappa$ is convex, the line segment $\theta v + (1-\theta)w$
  lies entirely within this ball for all $\theta \in [0,1]$.
  By the Fundamental Theorem of Calculus, we have
  \begin{equation*}
    g(t,v) - g(t,w) = \int_0^1 g_u\big(t, w + \theta(v-w)\big) (v-w) \diff{\theta},
  \end{equation*}
  where the spatial gradient $g_u$ corresponds to the second-order partial
  derivatives $f_{tu}$ (if $g=f_t$) or $f_{uu}$ (if $g=f_u$).
  Taking the norm and applying Assumption~\ref{ass:f2_bound},
  there exists a constant $C>0$ that depends on the dimension $d$ of the underlying space such that
  \begin{equation*}
    \|g(t,v) - g(t,w)\| \leq  \int_0^1 C \gamma_{\kappa}(t) \|v-w\| \diff{\theta}
    =  C \gamma_{\kappa}(t) \|v-w\|
  \end{equation*}
  for every $t \in [0,T]$.
  Summing the bounds for $f_t$ and $f_u$ yields \eqref{eq:f1_spatial} with
  $L_{\kappa}(t) := 2 C \gamma_{\kappa}(t)$, which is $L^2$-integrable since
  $\gamma_{\kappa} \in L^2(0,T)$.

  To prove \eqref{eq:f1_temporal}, let $s, t \in [0,T]$ with $s \le t$.
  Using the Fundamental Theorem of Calculus with respect to
  the temporal variable, we obtain
  \begin{equation*}
    g(t,v) - g(s,v) = \int_s^t g_t(\tau, v) \diff{\tau}.
  \end{equation*}
  Here, the partial derivative $g_t$ corresponds to the second-order
  partial derivatives $f_{tt}$ (if $g=f_t$) or $f_{ut}$ (if $g=f_u$).
  Taking the norm, applying Assumption~\ref{ass:f2_bound},
  and utilizing the Cauchy--Schwarz inequality, there again exists a constant $C>0$ that depends on the dimension $d$ such that
  \begin{align*}
    \|g(t,v) - g(s,v)\|
    \leq \int_s^t \| g_t(\tau, v)\| \diff{\tau}
    \leq  \int_s^t C \gamma_{\kappa}(\tau) \cdot 1 \diff{\tau}
    \leq
    C \|\gamma_{\kappa}\|_{L^2(0,T)} |t-s|^{\frac{1}{2}}.
  \end{align*}
  Summing the respective bounds for $f_t$ and $f_u$ yields the second assertion
  with the uniform constant $C_{\kappa} := 2 C \|\gamma_{\kappa}\|_{L^2(0,T)}$.

  It remains to show that $g$ is jointly continuous on $[0,T] \times \R^d$.
  Due to \eqref{eq:f1_spatial}, $g(t, \cdot)$ is Lipschitz continuous
  (and hence continuous) for all $t \in \mathcal{T}$, where $\mathcal{T} \subset [0,T]$
  is a set of full measure.
  Since $\mathcal{T}$ has full measure, it is dense in $[0,T]$.

  Let $t \in [0,T]$ be arbitrary and choose a sequence $(s_n)_{n \in \N}
  \subset \mathcal{T}$ such that $s_n \to t$.
  Due to \eqref{eq:f1_temporal}, we have the uniform bound
  \begin{equation*}
    \sup_{v \in \R^d, \|v\|\leq \kappa} \|g(t, v) - g(s_n, v)\| \le C_{\kappa} |t - s_n|^{\half}.
  \end{equation*}
  As $n \to \infty$, the right-hand side vanishes, implying that the sequence of
  spatially continuous functions $g(s_n, \cdot)$ converges uniformly to
  $g(t, \cdot)$ on $B_{\kappa}(0)$,
  where $B_{\kappa}(0) \subset \R^d$ is the closed ball of radius $\kappa$.
  Since the uniform limit of continuous functions is continuous, we
  conclude that $g(t, \cdot)$ is continuous on $B_{\kappa}(0)$ for every $t \in [0,T]$.

  To establish continuity, let $(t,v) \in [0,T] \times B_{\kappa}(0)$ and let
  $(s, w) \to (t,v)$. By the triangle inequality,
  \begin{equation*}
    \|g(t,v) - g(s,w)\| \leq \|g(t,v) - g(t,w)\| + \|g(t,w) - g(s,w)\|.
  \end{equation*}
  The first term vanishes as $w \to v$ because $g(t, \cdot)$ is continuous.
  The second term is bounded by $C_{\kappa} |t - s|^{\half}$, which
  vanishes as $s \to t$ uniformly with respect to $w$.
  Thus, $g$ is continuous on $[0,T] \times B_{\kappa}(0)$.
  Since $\kappa$ was arbitrary, continuity holds on $[0,T] \times \R^d$.
  From the continuity of $f_t$ and $f_u$, it follows that $f$ is locally
  Lipschitz continuous.
\end{proof}

\section{A randomized implicit Runge--Kutta method}
\label{sec:method}
First, we recall that a \emph{Runge--Kutta method} (or \emph{RK method}) with
$s \in \N$ stages is a one-step method of the general form
\begin{align}
  \label{eq:RKdef}
  \begin{split}
    u^{n+1} &= u^n + h \sum_{i = 1}^s b_i k_{n+1}^{i}\\
    k_{n+1}^{i} &= f\Big(t_n+c_i h, u^{n} +
    h\sum_{j = 1}^s a_{i,j} k_{n+1}^{j}\Big),
    \quad \text{for} \; i \in \{ 1,\ldots,s\}.
  \end{split}
\end{align}
Hereby, the coefficient vectors $b = (b_i)_{i\in \{1,\dots,s\}} \in \R^s$, $c = (c_i)_{i\in \{1,\dots,s\}} \in \R^s$, and the matrix
$A = (a_{i,j})_{i,j=1,\ldots,s} \in \R^{s,s}$ are called
the weights, the nodes, and the Runge--Kutta matrix of the scheme
\eqref{eq:RKdef}, respectively.
The real parameter $h > 0$ denotes the step size.
Moreover, the RK method \eqref{eq:RKdef} is fully characterized by its
Butcher tableau
\begin{equation}
  \label{eq:ButcherTab}
  \begin{array}{c|c}
    c & A   \\ \hline
      & b^\top
  \end{array}\;.
\end{equation}
The method is an \emph{implicit RK method} if $A$
has an entry $a_{i,j} \neq 0$ with $i \le j$.
In this paper, we focus on the family of \emph{singly diagonally implicit
RK methods} with $s = 2$ stages represented by the Butcher tableau
\begin{equation}\tag{Det}
  \label{scheme:det_fam_SDIRK}
  \begin{array}{c|cc}
    x   & x  & 0   \\
    1-x & 1-2x  & x \\ \hline
    & \frac12 & \frac12
  \end{array}
\end{equation}
with a fixed parameter $x\in[\half,1]$.
It will be convenient to write the method \eqref{scheme:det_fam_SDIRK} in the
form
\begin{align*}
  u^{n+1} = \Phi(t_n, h, x, u^n), \quad n \in \{0,1,2,\ldots\}
\end{align*}
also highlighting the dependence of $u^{n+1}$ on the choice of the parameter $x$.
Hereby, the one-step operator
$\Phi \colon [0,T] \times \R^+ \times [\frac{1}{2},1]\times \R^d \to \R^d$
is defined by
\begin{align}
  \label{eq:Phidef}
    \Phi(t,h,x,\genvec) &= \genvec + \frac{h}{2} \big(\kone(t,h,x,\genvec)+ \ktwo(t,h,x,\genvec)\big),
\end{align}
for $(t,h,x,\genvec) \in [0,T]\times \R^+ \times [\frac{1}{2},1] \times \R^d$,
where $\kone(t,h,x,\genvec), \ktwo(t,h,x,\genvec) \in \R^d$ are
the solutions to the implicit equations
\begin{align}
  \label{eq:kdef}
  \begin{split}
    \kone(t,h,x,\genvec) &= f\big(t+h x, \genvec+h x \kone(t,h,x,\genvec)\big),\\
    \ktwo(t,h,x,\genvec) &= f\big(t+h (1-x), \genvec+h (1-2x) \kone(t,h,x,\genvec)+h x \ktwo(t,h,x,\genvec)\big).
  \end{split}
\end{align}

\begin{remark}
  \label{rem:order_conds}
  The family of Runge--Kutta schemes \eqref{scheme:det_fam_SDIRK} has
  been studied extensively; see \cite[Chapter~IV.6]{Hairer1996}.
  In particular, for every fixed choice of $x \in [\half, 1]$ the method is
  convergent of order two. Indeed, the scheme \eqref{scheme:det_fam_SDIRK}
  satisfies the first and second-order conditions given by
  \begin{align*}
    \sum_{i=1}^2 b_i = \frac{1}{2} + \frac{1}{2} = 1
    \quad \text{and} \quad \sum_{j=1}^{2}\sum^2_{k=1}b_ja_{jk}
    = \frac{1}{2} (x + 0) + \frac{1}{2} (1 - 2x + x)
    = \half,
  \end{align*}
  respectively.
  Further, we have $c_i = \sum_{j=1}^2 a_{ij}$ for every $i \in \{1,2\}$.

  Moreover, for the particular value $x = \frac{1}{2}(1+\frac{\sqrt{3}}{3})$,
  the method is even convergent of order three.
  In fact, the third-order conditions are given by
  \begin{align*}
    \sum_{i=1}^{2} b_i c_i^2
    = x^2 - x + \frac{1}{2} \stackrel{!}{=} \frac{1}{3},
    \quad \text{and} \quad
    \sum_{i=1}^{2} \sum_{j=1}^{2} b_i a_{ij} c_j
    = x - x^2 \stackrel{!}{=} \frac{1}{6}.
  \end{align*}
  \qed
\end{remark}

After these preparations, we are in a position to introduce the
\emph{randomized diagonally implicit RK method} with $s = 2$ stages studied
throughout this paper.
To this end, let $\xi$ be a random variable with uniform distribution $\xi \sim
\mathcal{U}(\frac{1}{2},1)$. Then, we consider the RK method determined by the
Butcher tableau
\begin{equation}\tag{Rand}
  \label{scheme:SDIRK}
  \begin{array}{c|cc}
    \xi   & \xi  & 0   \\
    1-\xi & 1-2\xi  & \xi \\ \hline
    & \frac12 & \frac12
  \end{array}.
\end{equation}
In other words, our scheme \eqref{scheme:SDIRK} is a randomized version of
\eqref{scheme:det_fam_SDIRK} where the parameter $x$ is
replaced by the random variable $\xi$. Consequently, most of the entries of the
Runge--Kutta matrix and all of the nodes become random variables.
This also implies that the deterministic exact solution $u(t_n)$ is
approximated by a random variable $u^n$.

For the implementation and for the error analysis of the proposed randomized
scheme \eqref{scheme:SDIRK}, it is of most importance that an
independently generated sample of the random variable $\xi \sim
\mathcal{U}(\frac{1}{2},1)$ is used in every time step. This is summarized in the following assumption.

\begin{assumption}\label{ass:rand_var}
	Let $(\Omega,\F,(\F_{n})_{n\in\N},\P)$ be a complete,
	filtered probability space. Moreover, let $(\xi_n)_{n\in \N}$ be a family of independent, uniformly distributed random variables with values in $[\frac{1}{2},1]$ such that $\xi_{n}$ is $\F_{n}$-measurable, but is independent of $\F_{n-1}$ for every $n \in \N$.
\end{assumption}

Note that $\F_n$ can, for example, be defined as the natural filtration $\F_n = \sigma(\xi_1, \ldots, \xi_n)$.
Then the one-step method associated to
the Butcher tableau \eqref{scheme:SDIRK} is written as
\begin{align*}
  u^{n+1} = \Phi(t_n, h, \xi_{n+1}, u^n), \quad n \in \{0,1,2,\dots,N-1\},
\end{align*}
where $\Phi$ is the (deterministic) one-step operator defined in
\eqref{eq:Phidef}.
In Section~\ref{sec:existence} we show that this method is well-defined if
Assumption~\ref{ass:f_low} is satisfied.

\begin{remark}
  \label{rem:order_conds2}
  In light of Remark~\ref{rem:order_conds}, the second-order conditions
  are satisfied by the randomized scheme \eqref{scheme:SDIRK} for every
  realization of $\xi$ in $[\half, 1]$.
  Moreover, while the third-order conditions are not fulfilled for almost every
  realization of $\xi$, they hold true in expectation.
  More precisely, for $\xi \sim \mathcal{U}(\frac{1}{2},1)$, it follows that
  $\E[\xi^2 - \xi + \frac{1}{2}] = \frac{1}{3}$
  and $\E[\xi - \xi^2] = \frac{1}{6}$.
  \qed
\end{remark}

In the next section, we will show that our method \eqref{scheme:SDIRK} is
A-stable.
It is therefore not necessary to impose a step size restriction for stability
reasons when dealing with stiff problems. However, we will impose a mild
step size restriction to ensure that the implicit equations \eqref{eq:kdef}
are solvable.
To be more precise, if $\nu \in \R$ is the one-sided Lipschitz condition from
Assumption~\ref{ass:f_one_lip_B}, we assume that the time step $h$ is in the
interval $[0,h_{\max}]$, where $h_{\max}$ is given by
\begin{equation}
  \label{eq:def_hmax}
  h_{\max}=
  \begin{cases}
    \min \{\alpha, T \} & \text{if } \nu>0\\
    T & \text{if } \nu\leq 0
  \end{cases}
\end{equation}
for $\alpha<\nu^{-1}$.
Such a step size restriction is not uncommon for implicit RK-methods, compare
\cite[Theorem~IV.14.2]{Hairer1996}.

\section{Existence and measurability of the numerical solution}
\label{sec:existence}

The main goal of this section is to prove that the randomized SDIRK method
\eqref{scheme:SDIRK} is well-defined.
Due to the randomness introduced through the method, the
numerical approximation also depends on the family of random parameters
$(\xi_n)_{n\in \N}$.

We begin the section with two auxiliary results that we will use to prove the
existence of a measurable approximation.
The first lemma will be used to guarantee the existence of a solution
to the implicit equations in \eqref{eq:kdef} required for the definition
\eqref{eq:Phidef} of the one-step operator $\Phi$ of the RK method.
The second will ensure that the numerical solution defined by
\eqref{scheme:SDIRK} is measurable with respect to the random parameter.

\begin{lemma}
	\label{lem:circ_exist}
	For $R\in(0,\infty)$, let $F$ be defined on the ball $B_R(0) \subseteq
	\R^d$ around the origin with radius $R$ and have values in $\R^d$.
	Additionally, let $F$ be continuous and fulfil the condition
	\begin{equation} \label{eq:posF_lem_ex}
		\inner{F(v)}{v}\geq 0 \quad \text{for every $v$ in the boundary
			$\partial B_R(0)$ of } B_R(0).
	\end{equation}
	Then there exists at least one $v\in\R^d$ with $\|v\| \le R$
	such that $F(v)=0$.
\end{lemma}

A proof can be found in \cite[Section~9.1]{Evans2010}.

\begin{lemma}
	\label{lem:meas}
	Let $(\Omega, \widehat{\F}, \P)$ be a complete
	probability space.
	Let $\Lambda \colon \Omega \times \R^d \to \R^d$ be such that the
	following conditions are fulfilled:
	\begin{enumerate}
		\item The function $v\mapsto \Lambda(\omega,v)$ is continuous for every
		$\omega\in \Omega$.
		\item The function $\omega\mapsto \Lambda(\omega,v)$ is
		$\widehat{\mathcal{F}}$-measurable for every $v\in\R^d$.
		\item For every $\omega\in \Omega$, there exists a unique root of the
		function $\Lambda(\omega,\cdot)$.
	\end{enumerate}
	Define the function
	\begin{equation*}
		r \colon \Omega \rightarrow\R^d,\quad \omega\mapsto r(\omega),
	\end{equation*}
	where $r(\omega)$ is the unique root of $\Lambda(\omega,\cdot)$.
	Then $r$ is $\widehat{\mathcal{F}}$-measurable.
\end{lemma}

For a proof, we refer to \cite[Lemma~4.3]{Eisenmann2019}.
We can now combine these auxiliary results.

\begin{lemma}
	\label{lem:exist_measble}
	Let Assumption~\ref{ass:f_low} be fulfilled and let $(\Omega,
	\widehat{\F}, \P)$ be a complete probability space.
	Let $t \in [0,T]$ and $h \in [0,\min(T-t,h_{\max})]$ be given.
	Moreover, let $\xi$ and $\randvec$ be $\widehat{\mathcal{F}}$-measurable
	random variables with values in $[\half,1]$ and $\R^d$, respectively.
	Then the functions $\omega \mapsto \kone(t,h,\xi(\omega),\randvec(\omega))$
	and $\omega \mapsto \ktwo(t,h,\xi(\omega),\randvec(\omega))$ are well-defined
	and $\widehat{\mathcal{F}}$-measurable.
	Thus, the function $\omega \mapsto \Phi(t,h,\xi(\omega),\randvec(\omega))$ is
	also well-defined and $\widehat{\mathcal{F}}$-measurable.
	Furthermore, for every $K\ge0$ and $\omega \in \Omega$ with
	$\|\randvec(\omega)\|\leq K$, there exist $M_{K_1}, M_{K_2}>0$
	(compare Assumption~\ref{ass:f_low}) such that
	\begin{align}\label{eq:k12_bound}
		\begin{split}
			\|\kone(t,h,\xi(\omega),\randvec(\omega)) \|
			& \leq (1- h_{\max} \max(0, \nu))^{-1} M_{K_1},\\
			\|\ktwo(t,h,\xi(\omega),\randvec(\omega)) \|
			& \leq (1- h_{\max} \max(0, \nu))^{-1} M_{K_2},\\
		\end{split}
	\end{align}
	where $K_1 = K$ and $K_2 = K_1 + h_{\max}(1-h_{\max}\max(0,\nu))^{-1} M_{K_1}$.
\end{lemma}

\begin{proof}
	Let $K \ge 0$ and $\omega \in \Omega$ be such that
	$\|\randvec(\omega)\| \le K$. We first show that there exists a uniquely
	defined $\kone = \kone(t,h,\xi(\omega),\randvec(\omega)) \in \R^d$ satisfying
	the first equation in \eqref{eq:kdef}.

	By definition, the vector $\kone$ is a root of the function
	\begin{equation}\label{eq:g_1}
		\genvec \mapsto g_1(t,h,\xi(\omega),\randvec(\omega);\genvec)
		:=\genvec-f(t+h \xi(\omega),\randvec(\omega)+h \xi(\omega) \genvec).
	\end{equation}
	We will prove the existence of $\kone$ by making use of
	Lemma~\ref{lem:circ_exist}. To this end, we verify that
	$\genvec \mapsto g_1(t,h,\xi(\omega),\randvec(\omega);\genvec)$ fulfills
	condition \eqref{eq:posF_lem_ex} for a suitable radius $R$.
	Applying the one-sided Lipschitz condition from
	Assumption~\ref{ass:f_one_lip_B}, we observe that
	\begin{align*}
		&\inner{g_1(t,h,\xi(\omega),\randvec(\omega);\genvec)}{\genvec}\\
		&\quad= \|\genvec\|^2 - \inner{f(\tone,\randvec(\omega)+h \xi(\omega) \genvec)
			- f(\tone,\randvec(\omega))}{\genvec}
		- \inner{f(\tone,\randvec(\omega))}{\genvec}\\
		&\quad \geq (1- h \xi(\omega)\nu) \|\genvec\|^2
		- \|f(\tone,\randvec(\omega))\|\|\genvec\|.
	\end{align*}
	Due to Assumption~\ref{ass:f_bound} and the bound $\|\randvec(\omega)\| \le K$,
	we have $\|f(\tone,\randvec(\omega))\| \le M_K$.
	Inserting the assumption $h \xi(\omega)\nu \leq h_{\max} \max(0,\nu) < 1$,
	it follows that
	\begin{equation*}
		\inner{g_1(t,h,\xi(\omega),\randvec(\omega);\genvec)}{\genvec}\geq 0
	\end{equation*}
	for every $\genvec\in\R^d$ with
	$\|\genvec\| \geq (1-h_{\max} \max(0,\nu))^{-1} M_K$.
	Thus, we can apply Lemma~\ref{lem:circ_exist} with
	the radius $R_1 := (1-h_{\max} \max(0,\nu))^{-1} M_K$.
	This verifies the existence of a root $\kone =
	\kone(t,h,\xi(\omega),\randvec(\omega)) \in \R^d$
	of the function $g_1$ defined in \eqref{eq:g_1} with
	\begin{equation*}
		\|\kone(t,h,\xi(\omega),\randvec(\omega))\|
		\leq (1- h_{\max} \max(0,\nu))^{-1} M_{K} = R_1.
	\end{equation*}
	With this at hand, we verify that the mapping
	$\omega \mapsto \kone(t,h,\xi(\omega),\randvec(\omega))$ is
	$\widehat{\mathcal{F}}$-measurable by an application of
	Lemma~\ref{lem:meas}.
	Since $f$ satisfies a Carath\'{e}odory condition, it follows that
	the mapping $\genvec \mapsto g_1(t,h,\xi(\omega),\randvec(\omega);\genvec)$ is
	continuous for every fixed $t \in [0,T]$, $h \leq \min(T-t, h_{\max})$,
	and $\omega \in \Omega$ with $\|\randvec(\omega)\|\le K$.
	Similarly, the mapping $\omega \mapsto g_1(t,h,\xi(\omega),\randvec(\omega);\genvec)$
	is measurable.

	Hence, it remains to show that the root $\kone$ of the function $g_1$
	is unique.
	Suppose there exist two different roots
	$\kone = \kone(t,h,\xi(\omega),\randvec(\omega))$
	and $\overline{\kone} = \overline \kone(t,h,\xi(\omega),\randvec(\omega))$.
	Then using the one-sided Lipschitz condition from
	Assumption~\ref{ass:f_one_lip_B}, we obtain
	\begin{align*}
		\|\kone - \overline \kone \|^2
		&= \inner{f(\tone,\randvec(\omega)+h \xi(\omega)\kone)
			- f(\tone ,\randvec(\omega)+h \xi(\omega) \overline \kone)}{\kone -
			\overline \kone}\\
		&\leq h \xi(\omega)\nu \|\kone - \overline \kone \|^2.
	\end{align*}
	Since $h \xi(\omega)\nu \leq h_{\max} \max(0,\nu) <1$, it follows that
	$\kone = \overline \kone$.
	Thus, the function $g_1$ defined in \eqref{eq:g_1} has a unique root
	$\kone(t,h,\xi(\omega),\randvec(\omega))$, which is well-defined.
	Therefore, an application of Lemma~\ref{lem:meas} shows that
	$\omega \mapsto \kone(t,h,\xi(\omega),\randvec(\omega))$ is
	$\widehat{\mathcal{F}}$-measurable.

	Next, we turn to the second equation in \eqref{eq:kdef}.
	The vector $\ktwo = \ktwo(t,h,\xi(\omega),\randvec(\omega))$ is a root of the
	function
	\begin{equation}
		\label{eq:g_2}
		\genvec\mapsto g_2(t,h,\xi(\omega),\randvec(\omega);\genvec)
		:=\genvec - f(t+h (1-\xi(\omega)),\randvec(\omega)
		+h (1-2\xi(\omega))\kone(t,h,\xi(\omega),
		\randvec(\omega))+\xi(\omega) h \genvec).
	\end{equation}
	To apply Lemma~\ref{lem:circ_exist}, we first use the one-sided
	Lipschitz condition from Assumption~\ref{ass:f_one_lip_B} to find
	\begin{align*}
		\inner{g_2(t,h,\xi(\omega),\randvec(\omega),\genvec)}{\genvec}
		&= \|\genvec\|^2 - \inner{f(\ttwo,\randvec(\omega)
			+h (1-2\xi(\omega))\kone +h \xi(\omega) \genvec)}{\genvec}\\
		&\quad + \inner{f(\ttwo,\randvec(\omega)+h (1-2\xi(\omega))\kone))}{\genvec}\\
		&\quad - \inner{f(\ttwo,\randvec(\omega)+h (1-2\xi(\omega))\kone))}{\genvec}\\
		&\geq (1- h \xi(\omega)\nu) \|\genvec\|^2
		- \|f(\ttwo,\randvec(\omega)+h (1-2\xi(\omega))\kone)\|\|\genvec\|.
	\end{align*}
	Due to Assumption~\ref{ass:f_bound}, the bounds $\|\randvec(\omega)\| \le K$
	and $\|\kone\|\le (1 - h_{\max} \max(0, \nu))^{-1} M_K$,
	we obtain
	$\|f(\ttwo,\randvec(\omega)+h (1-2\xi(\omega))\kone)\| \le M_{K_2}$
	with $K_2 = K + h_{\max} (1 - h_{\max} \max(0, \nu))^{-1} M_K$.

	Since $h \xi(\omega)\nu \leq h_{\max} \max(0, \nu) < 1$, it follows that
	\begin{equation*}
		\inner{g_2(t,h,\xi(\omega),\randvec(\omega),\genvec)}{\genvec}\geq 0,
	\end{equation*}
	for every $\genvec\in \R^d$ with $\|\genvec\| \geq
	(1- h_{\max} \max(0,\nu))^{-1} M_{K_2}$.
	Hence, we can again apply Lemma~\ref{lem:circ_exist} with the radius
	\begin{align*}
		R_2 := (1- h_{\max} \max(0,\nu))^{-1} M_{K_2}.
	\end{align*}
	and obtain the existence of a root $\ktwo(t,h,\xi(\omega),\randvec(\omega))$
	of the function $g_2$ defined in \eqref{eq:g_2}.
	In particular, we have
	\begin{align*}
		\|\ktwo(t,h,\xi(\omega),\randvec(\omega))\|
		&\leq R_2 = (1- h_{\max} \max(0,\nu))^{-1} M_{K_2}.
	\end{align*}
	The proof of uniqueness and $\widehat{\mathcal{F}}$-measurability of $\ktwo$
	follows along the same lines as the proof for $\kone$ and is omitted.

	Finally, from \eqref{eq:Phidef} it follows that the one-step operator
	$\Phi(t,h,\xi(\omega),\randvec(\omega))$ is a linear combination of
	$\randvec(\omega)$, $\kone$ and $\ktwo$, and hence well-defined and
	$\widehat{\mathcal{F}}$-measurable as well.
	This completes the proof of the lemma.
\end{proof}

\begin{theorem}
	Let Assumption~\ref{ass:f_low} and \ref{ass:rand_var} be fulfilled.
	For every $h = \frac{T}{N}$, $N \in \N$, with $h \in (0,h_{\max}]$ and $\Phi$ defined in \eqref{eq:Phidef},
	there exists a uniquely determined, $(\mathcal{F}_n)_{n=0}^N$-adapted sequence
	of $\R^d$-valued random variables $(u^n)_{n=0}^N$ such that
	\begin{align*}
		\begin{cases}
			u^{n+1} = \Phi(t_n, h, \xi_{n+1}, u^n), \quad n \in \{0,\dots, N-1\},\\
			u^0 = u_0.
		\end{cases}
	\end{align*}
\end{theorem}

\begin{proof}
	The assertion follows by induction on $n \in \{0,\ldots,N\}$.
	More precisely, it holds trivially for $n=0$ because the constant map
	$\omega \mapsto u_0$ is $\mathcal{F}_0$-measurable.
	Assuming the existence of a uniquely determined and $\mathcal{F}_{n}$-measurable
	random variable $u^{n}$ for some $n < N$, Lemma~\ref{lem:exist_measble}
	immediately yields a unique $\F_{n+1}$-measurable random variable $u^{n+1}$.
\end{proof}

\section{Stability of Runge--Kutta methods}
\label{sec:stability}

In this section, we derive the stability properties of the randomized RK method
\eqref{scheme:SDIRK}.
To this end, we first recall well-known stability concepts for (deterministic)
RK methods from \cite{Hairer1996}.

The starting point of the stability analysis is the \emph{stability function}
$R \colon \mathbb{C} \to \mathbb{C}$ of an RK method.
For a general RK method with $s$ stages and Butcher tableau
\eqref{eq:ButcherTab}, its stability function $R$ is defined by
\begin{align*}
  R(z) = 1 + z b^{\top} (I_s - z A)^{-1} \one,
\end{align*}
where $\one := (1,\ldots,1)^{\top} \in \R^s$ and $I_s$ is the identity matrix in
$\R^{s,s}$.
If one applies an RK method to \emph{Dahlquist's test equation}
\cite{Dahlquist1963}
\begin{align}
  \label{eq:Dahlquist}
  \begin{cases}
    u' = \lambda u, \quad t \in (0,\infty),\\
    u(0) = 1,
  \end{cases}
\end{align}
with $\lambda \in \C$, then the numerical approximation with step size $h$ is
given by the recursion
\begin{align}
  \label{eq:sequence_u_n}
  u^{n+1} = R(h \lambda) u^n, \quad n \in \N.
\end{align}
This leads to the first notion of stability introduced by Dahlquist.

\begin{definition}[A-stability]
  \label{def:A-stable}
  An RK method with stability function $R$ is called \emph{A-stable} if
  $|R(z)| \leq 1 $ for all $z \in \C^{-} = \{ z\in\C \mid \Re(z)\leq 0\}$.
\end{definition}

Therefore, an A-stable RK method has the property that if the exact solution
$u$ to \eqref{eq:Dahlquist} satisfies $u(t)$ is bounded as $t \to \infty$, then for the numerical solution, it also follows that $u^n$ is bounded as $n \to \infty$.
For further details on the notion of A-stability, we refer to
\cite[Section~IV.3]{Hairer1996}.

The second notion of stability, originating from \cite{1975-Butcher},
generalizes the concept of A-stability to nonlinear ODEs of the form
\eqref{eq:problem}.
We follow the presentation from \cite[Definition~IV.12.2]{Hairer1996}.

\begin{definition}[B-stability]
  \label{def:B-stable}
  An RK method \eqref{eq:RKdef} is called \emph{B-stable} if the fulfillment of
  Assumption~\ref{ass:f_one_lip_B2} with $\nu=0$ (i.e., $f$ is contractive)
  implies for every step size $h>0$ that
  \begin{equation*}
    \|u_1-\widehat u_1\| \leq \|u_0-\widehat u_0\|,
  \end{equation*}
  where $u_1$ and $\widehat u_1$ are the numerical approximations produced
  by the RK method \eqref{eq:RKdef} after one step, starting with initial
  conditions $u_0$ and $\widehat u_0$, respectively.
\end{definition}

The B-stability of an RK method is often verified by demonstrating
its algebraic stability, cf.\ \cite[Definition~IV.12.5]{Hairer1996}.

\begin{definition}[Algebraic stability]
  \label{def:Alg-stable}
  An RK method \eqref{eq:RKdef} is called \emph{algebraically stable} if
  $b_{i} \ge 0$ for all $i\in\{1,\dots,s\}$ and the matrix
  $\mathbf{M} \in \R^{s,s}$ given by
  \begin{equation*}
    (\mathbf{M})_{ij}= b_i a_{ij} + b_{j} a_{ji} - b_i b_j,
    \quad i,j \in \{1,\dots,s\},
  \end{equation*}
  is positive semi-definite.
\end{definition}

\begin{lemma}
  \label{lemma:Algebraic_implies_B}
  Let Assumption~\ref{ass:f_low} hold. Then, the following implications hold
  for every RK method:
  \begin{equation*}
    \text{Algebraically stable}\implies\text{B-stable}\implies\text{A-stable}.
  \end{equation*}
\end{lemma}

The first implication in the lemma is proved in \cite[Theorem~IV.12.4]{Hairer1996}.
The second follows from \cite[Theorem~IV.12.11]{Hairer1996}.
In the next step, we verify that the deterministic RK method
\eqref{scheme:det_fam_SDIRK} is algebraically stable.
While our randomized scheme will ultimately restrict the parameter to $x \in
[\frac{1}{2}, 1]$, the following lemma establishes algebraic stability for a
considerably broader range of the parameter $x \in [\frac{1}{4}, \infty)$.

\begin{lemma}
  \label{lemma:choice_a_SDIRK}
  For every fixed value $x \in [\frac{1}{4},\infty)$, the RK method
  \eqref{scheme:det_fam_SDIRK} is algebraically stable.
\end{lemma}

\begin{proof}
  For the RK method \eqref{scheme:det_fam_SDIRK}, the matrix $\mathbf{M}$
  from Definition~\ref{def:Alg-stable} takes the form
  \begin{equation*}
    \mathbf{M} =\begin{pmatrix}
      x-\frac{1}{4} & \frac{1}{2}(1-2x)-\frac{1}{4}\\
      \frac{1}{2}(1-2x)-\frac{1}{4} & x - \frac{1}{4}
    \end{pmatrix}.
  \end{equation*}
  To verify that the matrix is positive semi-definite, we apply Sylvester's
  criterion. The determinant of the first principal minor $x-\frac{1}{4}$ is
  non-negative if $x\geq\frac{1}{4}$.
  Moreover, the determinant of $\mathbf{M}$ is
  \begin{align*}
    \mathrm{det}(\mathbf{M})
    = \Big(x-\frac{1}{4}\Big)\Big(x-\frac{1}{4}\Big)
    -\Big(\frac{1}{2}(1-2x)-\frac{1}{4}\Big)^2
    = \Big(x-\frac{1}{4}\Big)^2
    -\Big(\frac{1}{4} - x\Big)^2= 0,
  \end{align*}
  which is always non-negative.
  Thus, under the condition $x \ge \frac{1}{4}$, the matrix $\mathbf{M}$ is
  positive semi-definite.
\end{proof}

\begin{remark}[Direct proof of A-stability]
  Following \cite[conditions (IV~3.6) and (IV~3.7)]{Hairer1996}),
  A-stability of \eqref{scheme:det_fam_SDIRK} can also be shown directly.
  In fact, the associated stability function is given by
  \begin{equation*}
    R(z) = \frac{1+\frac{1}{2} (2-4 x) z+\frac{1}{2}
    \left(2 x^2-4 x+1\right)z^2}{(1 - x z)^2} =: \frac{P(z)}{Q(z)}.
  \end{equation*}
  Note that $R$ is analytic on $\C^{-}$ for every $x \ge 0$.
  Moreover, $|R(\mathrm{i} t )|\le 1$ for all $t\in\mathbb{R}$
  if and only if
  \begin{equation*}
    Q(\mathrm{i} t)Q(-\mathrm{i} t) - P(\mathrm{i} t) P(-\mathrm{i} t)
    = \frac{1}{4} t^4 (1-2x)^2 (4x-1) \ge 0.
  \end{equation*}
  By the maximum principle, A-stability is hence obtained for all
  $x \in [\frac14, \infty)$.
  Correspondingly, the stability domain includes the left half-plane,
  as illustrated in Figure~\ref{fig:stabplot}.
\end{remark}

\begin{figure}[htp]
  \label{fig:stabplot}
  \begin{center}
    \includegraphics[width=0.4\hsize]{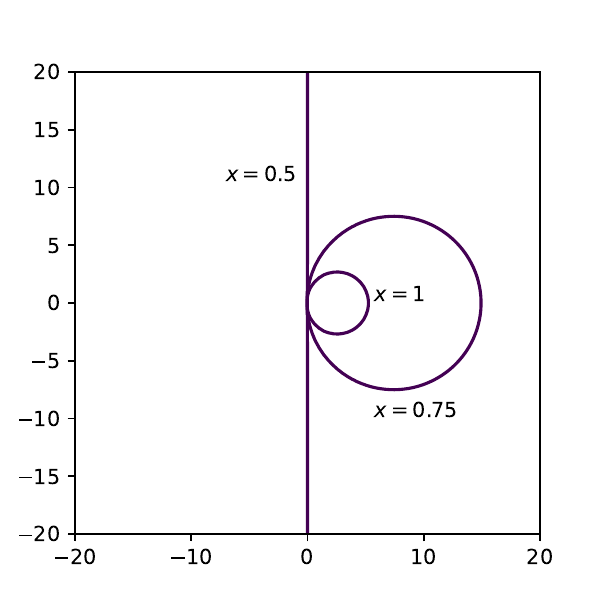}
  \end{center}
  \caption{Boundaries of the stability regions
    of scheme~\eqref{scheme:det_fam_SDIRK}.}
\end{figure}

The concept of algebraic stability is not only applicable when $f$ fulfills
Assumption~\ref{ass:f_one_lip_B2} with $\nu = 0$ but also if $\nu$ is
positive.
Indeed, if the Butcher matrix $A$ is invertible, it also gives a stability
result for $f$ when it satisfies a more general one-sided Lipschitz condition
with $\nu > 0$.
This notion of stability is called \emph{C-stability}, see
\cite[Satz~8.4.3]{Strehmel2012}.
For the deterministic RK method \eqref{scheme:det_fam_SDIRK} an
important stability result similar to C-stability holds.

\begin{lemma}\label{lem:alg_bound}
  Let Assumption~\ref{ass:f_low} be fulfilled and let $\Phi$ be defined as in
  \eqref{eq:Phidef}. Then there exists a constant $C_1>0$, independent of
  $f$ apart from $\nu$, such that
  \begin{equation*}
    \|\Phi(t,h,x,\genvec)-\Phi(t,h,x,\widehat \genvec)\|^2
    \leq (1+hC_1 \nuplus )\| \genvec -\widehat \genvec\|^2
  \end{equation*}
  is fulfilled for all $h \in [0, h_{\max}]$, all $t \in [0, T-h]$ and
  all $\genvec, \widehat \genvec \in \R^d$.
\end{lemma}

\begin{proof}
  In \cite[Proof of Proposition~IV.15.2]{Hairer1996}, the authors prove the
  result for any admissible Runge--Kutta method.
  Following the presented proof carefully and using
  \cite[Theorem~IV.14.6]{Hairer1996}, one can show that there exists a $C_1>0$,
  independent of the value of $x \in [\half,1]$, such that the stability bound
  holds.
\end{proof}

With the properties of the deterministic method established, we can now
summarize the stability of the randomized scheme \eqref{scheme:SDIRK}.
Since the random variable $\xi$ is distributed uniformly over an interval
where algebraic stability holds, the randomized method inherits these properties.

\begin{prop}
  \label{thm:randomized_stability}
  Let $\xi \sim U(\frac{1}{2}, 1)$. Then, the randomized RK method
  \eqref{scheme:SDIRK} is algebraically stable, and consequently B-stable and
  A-stable, with probability one.
  Furthermore, under Assumption~\ref{ass:f_low}, the stability bound
  \begin{equation*}
    \|\Phi(t,h,\xi,\genvec)-\Phi(t,h,\xi,\widehat \genvec)\|^2
    \leq (1+hC_1 \nuplus )\| \genvec -\widehat \genvec\|^2
  \end{equation*}
  holds for all $h \in [0, h_{\max}]$, all $t \in [0, T-h]$ and
  all $\genvec, \widehat \genvec \in \R^d$.
\end{prop}

We close this section with a brief comment on related stability concepts
discussed in \cite{Bochacik2021}. There, the authors consider three stability
regions for randomized RK methods, given by
\begin{align*}
  &\mathcal{R}_{MS} = \Big\{z \in \mathbb{C} \colon \prod_{k=1}^n R_k(z) \to 0 \ \text{in $L^2(\Omega)$ as} \ n\to \infty\Big\}, \\
  &\mathcal{R}_{AS} = \Big\{z \in \mathbb{C} \colon \prod_{k=1}^n R_k(z) \to 0 \ \text{almost surely as} \ n\to \infty\Big\}, \\
  &\mathcal{R}_{SP} = \Big\{z \in \mathbb{C} \colon \prod_{k=1}^n R_k(z) \to 0 \ \text{in probability as} \ n\to \infty\Big\},
\end{align*}
where $R_k(z)$ denotes the associated stability function of the randomized
Runge--Kutta method at the $k$-th step.

These concepts originate from an adoption of the notion of A-stability to
numerical methods for stochastic differential equations,
cf.~\cite{Bochacik2021,higham2000,saito1996}.
A randomized Runge--Kutta method is then called,
for instance, \emph{mean-square stable} if $\C^- \setminus i \R \subseteq \mathcal{R}_{MS}$.
Compared to Definition~\ref{def:A-stable}, these stability concepts are weaker.
They only guarantee that the numerical solution decays asymptotically in a
stochastic sense, which theoretically allows the numerical approximation to grow
during individual time steps.

In contrast to this, our randomized RK method \eqref{scheme:SDIRK} satisfies
$|R(z; \xi)| < 1$ for every $z \in \mathbb{C}^- \setminus i \R$ and with probability one
for every realization of $\xi \sim U(\half,1)$.
Therefore, the corresponding random sequence \eqref{eq:sequence_u_n}
generated by \eqref{scheme:SDIRK} possesses the stronger property
$|u^{n+1}| < |u^n|$ with probability one.

The last step is to ensure that our numerical method is uniformly bounded.

\begin{lemma}\label{lem:apriori_bound}
  Let Assumptions~\ref{ass:f_low} and \ref{ass:rand_var} be fulfilled.
  The solution $(u^n)_{n=0}^N$ of
  \begin{align*}
    \begin{cases}
      u^{n+1} = \Phi(t_n, h, \xi_{n+1},u^n), \quad n \in \{0,\dots, N-1\},\\
      u^0 = u_0,
    \end{cases}
  \end{align*}
  with $\Phi$ defined in \eqref{eq:Phidef}, is uniformly bounded for every $h = \frac{T}{N}$, $N \in \N$, with $h \in (0,h_{\max}]$.
  More precisely, there exists a constant $M_{\mathrm{apriori}}>0$ such that for
  every $h \in (0,h_{\max}]$
  \begin{equation*}
    \|u^n(\omega)\|\leq M_{\mathrm{apriori}} \quad \text{\ for all }
    n\in\{0,1,\dots,N\}\text{ and }\omega\in\Omega
  \end{equation*}
  is fulfilled.
\end{lemma}

\begin{proof}
  Applying Lemma~\ref{lem:alg_bound}, it follows for every
  $v\in\R^d$ and every $x\in[\half,1]$ that
  \begin{align*}
    \|\Phi(t,h,x,v)\|
    &\leq \|\Phi(t,h,x,v)-\Phi(t,h,x,0)\|
    +\|\Phi(t,h,x,0)\|\\
    &\leq (1+hC_1\nuplus)^{\half}\|v\|
    +\|\Phi(t,h,x,0)\|\\
    &\leq (1+hC_1\nuplus)\|v\|+\|\Phi(t,h,x,0)\|,
  \end{align*}
  where we used $(1+hC_1\nuplus)^{\half}\leq 1+hC_1\nuplus$.
  The bounds on $\kone$ and $\ktwo$ from \eqref{eq:k12_bound}, can be used to
  show
  \begin{align*}
    \|\Phi(t,h,x,0)\|
    \leq \frac{h}{2}\|\kone(t,h,x,0)\|
    +\frac{h}{2}\|\ktwo(t,h,x,0)\|
    \leq \frac{h}{2}(1-h_{\max} \max(0,\nu))^{-1} \big( M_{K_1} + M_{K_2}\big)
  \end{align*}
  for $M_{K_1}$ and $M_{K_2}$ from Assumption~\ref{ass:f_bound} with $K_1=0$ and
  $K_2=h_{\max}(1-h_{\max}\max(0,\nu))^{-1}M_{K_1}$. Thus, it follows that
  \begin{equation*}
  	\|\Phi(t,h,x,v)\|\leq (1+hC_1\nuplus) \| v\| + \frac{h}{2}(1-h_{\max}\nuplus)^{-1}
  	\big( M_{K_1} + M_{K_2}\big).
  \end{equation*}
  Applying this to the numerical solution yields
  \begin{align*}
    \|u^{n+1}(\omega)\|
    &= \| \Phi(t_n,h,\xi_{n+1}(\omega), u^n(\omega)) \|\\
    &\leq (1+hC_1\nuplus) \| u^n(\omega)\| + \frac{h}{2}(1-h_{\max}\nuplus)^{-1}
    \big( M_{K_1} + M_{K_2}\big)\\
    &\leq \| u_0\| + \frac{T}{2}(1-h_{\max} \max(0,\nu))^{-1} \big( M_{K_1} + M_{K_2}\big)
    + h C_1 \nuplus \sum_{k=0}^{n} \|u^k(\omega)\|
  \end{align*}
  for all $n\in\{0,1,\dots,N\}$ and $\omega\in\Omega$.
  The discrete Gr\"onwall inequality~\cite{Clark1987} then leads to the
  claimed result of this lemma.
\end{proof}

\section{Auxiliary error bounds}
\label{sec:taylor}

In this section, we establish the foundation for the error analysis
in the following section. We will give bounds that will be crucial
to bound the local error, the difference between our numerical
method and the exact flow after one single time step.
To this end, we derive Taylor approximations of both
the numerical method and the exact flow including suitable remainder terms.
For these Taylor approximations, the Hessian of $f$ with respect to~$u$ is
needed, which is a third-order tensor.
To avoid introducing notation for tensor products, we will use the shortened notation given by
\begin{equation}\label{eq:tensor_notation}
  (f_{uu}(t,u)[v][w])_i
  :=\sum_{j,k=1}^d \frac{\partial^2 f_i}{\partial x_j\partial x_k}(t,u)v_jw_k
  \quad \text{with }
  (f_{uu}(t,u)[v]^2)_i
  := (f_{uu}(t,u)[v][v])_i.
\end{equation}
We first define the exact flow $\Psi(v,t,h)$ at the time $t+h$ with initial
value $v$ at time $t$ as the solution of
\begin{equation} \label{eq:def_Psi}
  \begin{cases}
    \frac{\diff}{\diff h}\Psi(v,t,h)=f(t+h,\Psi(v,t,h)), &h\in(0,T-t),\\
    \Psi(v,t,0)=v.
  \end{cases}
\end{equation}
Note that $\Psi(u(t),t,h)=u(t+h)$ is fulfilled.
Furthermore, the existence of the solution $\Psi$ can be proven in the same way as
Lemma~\ref{lem:uexist}.
From \cite[Lemma~IV.12.1]{Hairer1996} and Lemma~\ref{lem:uexist}, it follows that
\begin{equation}\label{eq:boundPsi}
  \|\Psi(v,t,h)\|
  \leq \mathrm{e}^{\nu h} \|v\|
  + \int_{t}^{t+h} \mathrm{e}^{(t+h-\tau) \nu} \|f(\tau,0)\| \diff{\tau}
\end{equation}
for every $v \in \R^d$, $t \in [0,T]$ and $h \in (0,T-t]$.
For our analysis, we use the second-order Taylor approximation of $\Psi(v,t,h)$
with an integral remainder term.
Inserting the ODE \eqref{eq:def_Psi}, we find that
\begin{equation}
  \label{eq:Taylor_Psi}
  \begin{split}
  \Psi(\genvec,t,h)&=\Psi(\genvec,t,0)+h\Psi_h(\genvec,t,0)
  +\frac{h^2}{2}\Psi_{hh}(\genvec,t,0)
  +\frac{1}{2}\int_0^h (h-\tau)^2\Psi_{hhh}(\genvec, t, \tau)\diff{\tau}\\
  &=\genvec+ hf(t,\genvec)+ \frac{h^2}{2} (f_t(t,\genvec) +f_u(t,\genvec)f(t,\genvec))
  +\frac{1}{2}\int_0^h (h-\tau)^2F(t+\tau, \Psi(\genvec,t,\tau))\diff{\tau},
  \end{split}
\end{equation}
where we summarize all the second derivatives of $f$ through
\begin{equation}\label{eq:def_F}
  F(t,\genvec)
  = f_{tt}(t, \genvec)+2f_{tu}(t, \genvec)f(t, \genvec)
  + f_{uu}(t, \genvec)[f(t, \genvec)]^2
  + f_u(t, \genvec) f_t(t, \genvec)+ f_u^2(t, \genvec)f(t, \genvec).
\end{equation}
The continuity of $\Psi$ and its first two derivatives together with the
$L^2$-integrability of the third derivative follow from
Assumption~\ref{ass:f_high_conv}.
Therefore, the Taylor approximation is justified.
For a more convenient notation in the following proofs, we split up the
remainder term $F$ into two parts
\begin{equation}
  \label{eq:def_F1}
  F_1(t,\genvec) = f_{tt}(t, \genvec) + 2f_{tu}(t, \genvec)f(t, \genvec)
  +f_{uu}(t, \genvec)[f(t, \genvec)]^2,
\end{equation}
and
\begin{equation}\label{eq:def_F2}
  F_2(t,\genvec)
  =f_u(t, \genvec) f_t(t, \genvec)+ f_u^2(t, \genvec)f(t, \genvec).
\end{equation}
The remainder terms $F_1$ and $F_2$ correspond to one of the two third-order
conditions, compare with Remark~\ref{rem:order_conds}.
While $F_1$ collects all partial derivatives of order two, the function $F_2$
only contains partial derivatives up to order one.
The two remainder terms $F_1$ and $F_2$ need different treatments in the
forthcoming analysis.
Next, we state regularity results of the remainder terms $F_1$ and $F_2$.

\begin{lemma}\label{lem:lipschitzF}
  Let Assumption~\ref{ass:f_high_conv} be fulfilled.
  Then for every $K>0$, there exists a constant $C_K>0$ such that
  \begin{align*}
    \|F_1(t, v_1) - F_1(t, v_2)\|
    &\leq C_K L_{K}(t) \| v_1 - v_2 \|^{\sigma}
    + C_K \gamma_{K}(t) \| v_1 - v_2\|\\
    \|F_2(t_1, v_1) - F_2(t_2, v_2)\|
    &\leq C_{K} \big(|t_1 -t_2|^{\frac{1}{2}} + \tL_{K}(t_1)\| v_1 - v_2 \|
    \big)
  \end{align*}
  for all $t,t_1,t_2 \in [0,T]$ and $v_1,v_2 \in \R^d$ such that
  $\|v_1\|, \|v_2\| \leq K$, where $F_1$ and $F_2$ are given in
  \eqref{eq:def_F1} and \eqref{eq:def_F2}, while $\gamma_{K}$ and $L_{K}$
  are given in Assumption~\ref{ass:f2_bound} and Assumption~\ref{ass:f2_hold},
  respectively.
  Furthermore,
  \begin{equation*}
  	\| F(t,v)\| \leq C_K \gamma_K(t)
  \end{equation*}
  is fulfilled for all $t\in [0,T]$ and $v \in \R^d$ such that $\|v\| \leq K$.
\end{lemma}

\begin{proof}
  Let $K >0$ be arbitrary. For $F_1$, we apply the bound~\eqref{eq:f0_Lip} from Lemma~\ref{lem:implied_continuity},
  Assumption~\ref{ass:f2_bound},
  and Assumption~\ref{ass:f2_hold}. In addition, by using the fact that
  $t$ is in the compact set $[0,T]$ and $\|v_1\|$, $\|v_2\|$ are bounded by $K$, we obtain
  \begin{align*}
    \|F_1(t, v_1) - F_1(t, v_2)\|
    %
    %
    &\leq \| f_{tt}(t, v_1) - f_{tt}(t, v_2)\|
    + 2 \| f_{tu}(t, v_1)f(t, v_1)- f_{tu}(t, v_2)f(t, v_2)\|\\
    &\quad + \| f_{uu}(t, v_1)[f(t, v_1)]^2 - f_{uu}(t, v_2)[f(t, v_2)]^2 \|\\
    &\leq \| f_{tt}(t, v_1) - f_{tt}(t, v_2)\|
    +2 \| f_{tu}(t, v_1)\| \|f(t, v_1) - f(t, v_2)\|\\
    &\quad +2 \| f_{tu}(t, v_1)- f_{tu}(t, v_2)\| \|f(t, v_2)\|
    + \| f_{uu}(t, v_1) - f_{uu}(t, v_2) \| \|f(t, v_1)\|^2 \\
    &\quad +  \| f_{uu}(t, v_2)\| \| f(t, v_1) + f(t,v_2)\|
    \| f(t, v_1) - f(t, v_2) \|\\
    %
    %
    &\leq C_{K} \big(L_{K}(t) \| v_1 - v_2 \|^{\sigma}
    + \gamma_{K}(t) \| v_1 - v_2\|\big).
  \end{align*}
  Similarly, we now apply the bound~\eqref{eq:f1_spatial},
  \eqref{eq:f1_temporal} and \eqref{eq:f0_Lip} from Lemma~\ref{lem:implied_continuity} as well as by using the fact that
  $t_1$, $t_2$ are in the compact set $[0,T]$ and $\|v_1\|$, $\|v_2\|$ are bounded by $K$, we obtain for $F_2$ that
  \begin{align*}
    &\|F_2(t_1, v_1) - F_2(t_2, v_2)\|\\
    %
    %
    &\leq \| f_u(t_1, v_1) f_t(t_1, v_1) - f_u(t_2, v_2) f_t(t_2, v_2)\|
    + \| f_u^2(t_1, v_1)f(t_1, v_1) - f_u^2(t_2, v_2)f(t_2, v_2)\| \\
    &\leq \|f_u(t_1, v_1)\| \| f_t(t_1, v_1) - f_t(t_2, v_2)\|
    + \| f_u(t_1, v_1) - f_u(t_2, v_2) \| \|f_t(t_2, v_2) \| \\
    &\quad + \|f_u(t_1, v_1)\|^2 \| f(t_1, v_1) - f(t_2, v_2)\|
    + \|f_u(t_1, v_1)\| \| f_u(t_1, v_1) - f_u(t_2, v_2)\| \|f(t_2, v_2)\| \\
    &\quad + \| f_u(t_1, v_1) - f_u(t_2, v_2)\|
    \|f_u(t_2, v_2)\|  \| f(t_2, v_2) \|\\
    &\leq C_{K} \big(|t_1 -t_2|^{\frac{1}{2}} + \tL_K(t_1) \| v_1 - v_2 \| \big).
  \end{align*}
\end{proof}

\begin{lemma}
  \label{lem:fuest}
  Let Assumption~\ref{ass:f_high_conv} be fulfilled.
  Then for all $h>0$ and $\genvec,w\in\R^d$ it holds
  \begin{equation*}
    (1 -h \nu) \|w\|\leq \|(I- h f_u(t,\genvec))w\|.
  \end{equation*}
  In particular, this implies that
  \begin{equation*}
  	\|(I- h f_u(t,\genvec))^{-1}\|\leq (1 - h \nu)^{-1}
  \end{equation*}
  is fulfilled for every $h\in(0,h_{\max}]$, $t\in[0,T]$ and $\genvec\in \R^d$.
\end{lemma}

\begin{proof}
  Let $\varepsilon > 0$ be given.
  Applying the one-sided Lipschitz condition from
  Assumption~\ref{ass:f_one_lip_B2}, the second fundamental theorem of calculus
  and \eqref{eq:f1_spatial} from Lemma~\ref{lem:implied_continuity}, it follows that
  \begin{align*}
    \nu\varepsilon^2 \|{w}\|^2
    &\geq \inner{f(t,\genvec+\varepsilon w)-f(t,\genvec)}{\varepsilon w}\\
    &= \innerB{\varepsilon f_u(t,\genvec)w + \int_{0}^{\varepsilon}
    \big(f_{u}(t,\genvec+\tau w)w-f_u(t,\genvec)w \big)\diff{\tau}}{\varepsilon w}\\
    &\geq \varepsilon^2 \inner{f_u(t,\genvec)w}{w}
    - \varepsilon \int_0^{\varepsilon}\|f_{u}(t,\genvec+\tau w)-f_u(t,\genvec)\|\diff{\tau}\|w\|^2 \\
    &\geq \varepsilon^2 \inner{f_u(t,\genvec)w}{w}
    - \varepsilon L_{K}(t) \int_0^{\varepsilon} \| \tau w\|^{\frac{1}{2}} \diff{\tau}
    \|w\|^2 \\
    &= \varepsilon^2\inner{f_u(t,\genvec)w}{w}
    - \frac{2}{3} L_K(t) \varepsilon^{\frac 52} \|w\|^{\frac{5}{2}}
  \end{align*}
  for $K = \max_{\tau \in [0,\varepsilon]} \|\genvec+\tau w\|$.
  The function $L_K$ is finite almost everywhere in $[0,T]$ as it is $L^2$-integrable. Dividing by $\varepsilon^2$ and letting $\varepsilon \to 0$, the last term
  vanishes and we obtain
  $\nu \|w\|^2\geq \inner{f_u(t,\genvec)w}{w}$ for almost every $t\in[0,T]$.
  This then shows that
  \begin{equation*}
    (1-h\nu)\|w\|^2 \leq \inner{(I- h f_u(t,\genvec))w}{w}
    \leq \|(I- h f_u(t,\genvec))w\|\|w\|
  \end{equation*}
  for almost every $t\in[0,T]$.
  For $\|w\| \neq 0$, we divide by $\|w\|$ and obtain our claimed result. Additionally, the result holds trivially for $\|w\| = 0$. Since $f_u$ is continuous, we get that the result extends to all $t\in[0,T]$.
\end{proof}

The next step is to find an analogous expansion as in \eqref{eq:Taylor_Psi}
for the function $\Phi$ describing our one-step method from \eqref{eq:Phidef}.
To achieve this, we state the Taylor approximations of $\kone(t,h,x,v)$ and
$\ktwo(t,h,x,v)$ from \eqref{eq:kdef} with respect to the variable $h$.
As a first step, we will prove the existence of their derivatives and give
definitions through implicit equations for them in the following lemma.
For easier readability of the next lemma, the argument $(t,h,x,v) \in [0,T]
\times [0,h_{\max}] \times [\frac{1}{2}, 1] \times \R^d$ is dropped for the
functions $\ellone$, $\kone$, $\elltwo$ and $\ktwo$ and their partial
derivatives.
The elaborated bounds of $\koness$ and $\ktwoss$ are required for
Lemma~\ref{lem:E_norm_loc_cond}.

\begin{lemma}\label{lem:bounds}
  Let Assumption~\ref{ass:f_high_conv} be fulfilled.
  Then the partial derivatives of $k^1$ and $k^2$ with respect to $h$ are
  uniquely defined through the implicit equations
  \begin{align*}
    \kones
    &= x f_t(\ellonearg) + f_u(\ellonearg) \ellones \\
    \koness
    &= x^2 f_{tt}(\ellonearg) +2xf_{tu}(\ellonearg) \ellones
    + f_{uu}(\ellonearg)[\ellones]^2 + xf_u(\ellonearg) (2\kones+h\koness ) \\
    \ktwos
    &= (1-x) f_t(\elltwoarg) + f_u(\elltwoarg) \elltwos\\
    \ktwoss
    &= (1- x)^2 f_{tt}(\elltwoarg) + 2(1-x)f_{tu}(\elltwoarg)\elltwos
    + f_{uu}(\elltwoarg)[\elltwos]^2
    + (1-2x) f_u(\elltwoarg) ( 2 \kones + h \koness ) \\
    &\quad + x f_u(\elltwoarg) (2 \ktwos + h \ktwoss),
  \end{align*}
  with the abbreviations
  \begin{align*}
    &\theta^1:=t+hx ,&\qquad& \theta^2:= t+h(1-x),\\
    &\ellone:=\initvec+hx\kone,&\qquad
    &\elltwo := \initvec+h(1-2x)\kone +h x \ktwo,\\
    &\ellones:= x (\kone+h\kones),&\qquad&
    \elltwos:= (1-2x) (\kone + h \kones)+ x (\ktwo + h \ktwos).
  \end{align*}
  Furthermore, for every $K>0$, there exist a constant $\tK>0$ such that
  \begin{align*}
    \|\kone\| &\leq \tK, & \|\ktwo\| &\leq \tK, &
    \|\kones\| &\leq \tK, & \|\ktwos\| &\leq \tK, \\
    \|\ellone\| &\leq \tK, & \|\elltwo\| &\leq \tK,&
    \|\ellones\| &\leq \tK, & \|\elltwos\| &\leq \tK.
  \end{align*}
  Moreover, there exist a $C_K>0$ such that
  \begin{align*}
    \|k^1_{hh}\| \leq C_K (1+x\gamma_{\tK}(\theta^1)), \quad
    \|k^2_{hh}\| \leq C_K\big(1 + h \gamma_{\tK}(\theta^1)
    + (1-x+h) \gamma_{\tK}(\theta^2)\big)
  \end{align*}
  for every $t\in[0,T]$, $h\in[0,\min(T-t,h_{\max})]$,
  $x\in[\half,1]$, and $v\in\R^d$ with $\|v\|\leq K$.
\end{lemma}

\begin{proof}
  Let $K > 0$, $t\in[0,T]$, $h\in[0,\min(T-t,h_{\max})]$,
  $x\in[\half,1]$, and $v\in\R^d$ with $\|v\|\leq K$ be arbitrary.
  First we bound $\kone$ and $\ktwo$. In order to do this, we apply
  Lemma~\ref{lem:exist_measble}, where we replace the random variables $\xi$ and
  $\eta$ with $x$ and $v$, respectively. Thus, we have
  \begin{equation*}
    \|\kone\|\leq (1-h_{\max}\nuplus)^{-1}M_{K_1}
  \end{equation*}
  and
  \begin{equation*}
    \|\ktwo\|\leq (1-h_{\max}\nuplus)^{-1}M_{K_2},
  \end{equation*}
  where $K_1=K$ and $K_2=K+h_{\max}(1-h_{\max}\nuplus)^{-1}M_{K_1}$ with
  $M_{K_1}$ and $M_{K_2}$ from Assumption~\ref{ass:f_low}.
  Recall that $\ellone$ and $\elltwo$ are linear combinations of $v$, $\kone$
  and $\ktwo$. Thus, we have
  \begin{align*}
    \|\ellone\|
    \leq K + h_{\max} (1-h_{\max}\nuplus)^{-1} M_{K_1}
  \end{align*}
  and
  \begin{equation*}
    \|\elltwo\| \leq K + h_{\max} (1-h_{\max}\nuplus)^{-1}
    \big(M_{K_1} + M_{K_2} \big).
  \end{equation*}
  Preliminarily, we set $\tK$ equal to the maximum of all these bounds.

  The next step is to prove that the partial derivatives of $\kone$ and $\ktwo$
  with respect to $h$ exist. We apply the implicit function theorem to obtain
  the first partial derivatives. The second derivatives of the two functions can
  be obtained by a suitable derivation of the terms.
  We begin to provide the needed regularity for $\kone$. By definition, the
  vector $\kone(t,h,x,\initvec)$ is the unique root of the function $\genvec
  \mapsto g_1(t,h,x,\initvec;\genvec):= \genvec - f(t+h x, \initvec+h x
  \genvec)$.
  The Jacobian of $g_1$ with respect to~$\genvec$ is given by
  \begin{equation*}
    D_{\genvec}[g_1(t,h,x,\initvec;\genvec)]=I-hxf_u(t+h x, \initvec+h x\genvec).
  \end{equation*}
  From Lemma~\ref{lem:fuest}, it follows that the Jacobian
  $D_{\genvec}[g_1(t,h,x,\initvec;\genvec)]$ is injective and therefore also
  invertible as a $d \times d$-matrix.
  We apply the implicit function theorem (\cite[Theorem~15.1]{Deimling.1985}) with $t,x$ and
  $\initvec$ fixed. Thus, we obtain the local existence of a continuously
  differentiable function $h \mapsto \kone(t,h,x,v)$ that satisfies
  $g_1(t,h,x,\initvec;\kone(t,h,x,v))=0$.
  Since Lemma~\ref{lem:exist_measble} already verified existence of
  $\kone(t,\cdot,x,v)$ on the interval $[0,\min(T-t,h_{\max})]$, it follows
  that the function is uniquely defined on the interval $[0,\min(T-t,h_{\max})]$.
  The derivative is given by
  \begin{equation}\label{eq:k1s}
    \begin{split}
      \kones &= - \big(D_{\genvec}[g_1(t,h,x,\initvec; \kone)]\big)^{-1} D_{h}[g_1(t,h,x,\initvec;\kone)]\\
      &=\big(I-h xf_u(\ellonearg)\big)^{-1} \big(x f_t(\ellonearg)
      + x f_u(\ellonearg) \kone\big)
    \end{split}
  \end{equation}
  as a consequence of the implicit function theorem.
  To obtain the implicit equation that describes $\kones$, we multiply
  $\kones$ by $I-hxf_u(\ellonearg)$ and add $hxf_u(\ellonearg)\kones$ to both
  sides of the equation.
  To bound $\kones$, we apply the explicit expression from \eqref{eq:k1s} and
  the second bound from Lemma~\ref{lem:fuest} to obtain
  \begin{align*}
    \| \kones \|
    &= \big\|  \big(I-h xf_u(\ellonearg)\big)^{-1} \big(x f_t(\ellonearg)
    + x f_u(\ellonearg)  f(\ellonearg)\big) \big\|\\
    &\leq (1 - hx \nu)^{-1} \big( \|f_t(\ellonearg)\|
    + \|f_u(\ellonearg)\| \|f(\ellonearg)\| \big)\\
    &\leq (1 - h_{\max} \nuplus)^{-1} \big( \|f_t(\ellonearg)\|
    + \|f_u(\ellonearg)\| \|f(\ellonearg)\| \big).
  \end{align*}
  Since $f$ is continuously differentiable and all the arguments of $f$, $f_t$
  and $f_u$ lie in a compact set $[0,T] \times B_R(0)$ with radius $R$
  depending on $\tK$ from the established bound of $\ellone$,
  it follows that $\|\kones\|$ is also bounded by a constant
  depending purely on $K$.
  Writing out the partial derivative $\ellones$ of $\ellone$, we find
  \begin{align*}
    \|\ellones\| &=\|x(\kone+h\kones)\|\leq x \|\kone\| + hx \|\kones\|
    \leq \| f(\ellonearg) \| + h_{\max} \|\kones\|.
  \end{align*}
  Again, all terms are bounded by constants depending on $K$. By possibly
  enlarging $\tK$, we denote the maximum of all bounds again by $\tK$.

  Next, we show that $\koness$ exists and is bounded as claimed. Note that the implicit function theorem as stated in \cite[Theorem~15.1]{Deimling.1985} is not applicable to prove the existence of $\koness$, as this function is not continuous in its first argument.
  Here, we apply that for a matrix valued function $V_h \colon [0,T] \to \R^{d,d}$,
  its derivative of the inverse is given by $(V^{-1})_h(h)=-(V^{-1}V_hV^{-1})(h)$
  in the first step and insert \eqref{eq:k1s} in the second. Then we obtain
  \begin{align}\label{eq:k1ss}
    \begin{split}
      \koness
      &= - \big(I-hxf_u(\ellonearg)\big)^{-1} \big(I-hxf_u(\ellonearg)\big)_h
      \big(I-hxf_u(\ellonearg)\big)^{-1}\big(x f_t(\ellonearg) + x f_u(\ellonearg)
      \kone\big)\\
      &\quad + \big(I-hxf_u(\ellonearg)\big)^{-1} \big( xf_t(\ellonearg)
      + xf_u(\ellonearg)\kone\big)_h\\
      &= \big(I-hxf_u(\ellonearg)\big)^{-1} \big(xf_u(\ellonearg)
      + h x^2f_{tu}(\ellonearg)+ h x f_{uu}(\ellonearg)[\ellones, \cdot]\big)\kones\\
      &\quad + \big(I-hxf_u(\ellonearg)\big)^{-1} \big( x^2 f_{tt}(\ellonearg)
        + x f_{tu}(\ellonearg) \ellones + x^2 f_{tu}(\ellonearg)\kone\\
      &\hspace{4cm}+ x f_{uu}(\ellonearg)[\ellones,\kone]
      + xf_u(\ellonearg)\kones\big)\\
      & = \big(I-hxf_u(\ellonearg)\big)^{-1} \big(x^2 f_{tt}(\ellonearg)
      + 2x f_{tu}(\ellonearg)\ellones +  f_{uu}(\ellonearg)[\ellones]^2
      + 2x f_u(\ellonearg)\kones
      \big).
    \end{split}
  \end{align}
  For the implicit equation that describes $\koness$, we multiply the previous equation first by
  $I-hxf_u(\ellonearg)$ and then add $hxf_u(\ellonearg)\koness$ to both sides.
  To show the bound for the second partial derivatives of $\kone$, we apply the
  explicit definitions from \eqref{eq:k1ss} and insert the second
  bound from Lemma~\ref{lem:fuest} to find
  \begin{align*}
    \|\koness\|
    &\leq (1 - h x \nu)^{-1} \big( x^2 \|f_{tt}(\ellonearg)\|
    + 2 x \| f_{tu}(\ellonearg)\| \|\ellones\|
    +  \|f_{uu}(\ellonearg)\| \|\ellones\|^2
    + 2x\|f_u(\ellonearg)\| \|\kones\|
    \big).
  \end{align*}
  Due to the continuous differentiability of $f$ and $\|\ellone\| \le \tK$ we
  obtain that $\|f_u(\ellonearg)\|$ is bounded by a constant depending on $K$.
  For $\|f_{tt}(\ellonearg)\|$, $\| f_{tu}(\ellonearg)\|$, and
  $\|f_{uu}(\ellonearg)\|$, we apply Assumption~\ref{ass:f2_bound}
  for $\kappa = \tK$. Together this yields
  \begin{align*}
    \|\koness\|
    &\leq C_K \big( x^2 \|f_{tt}(\ellonearg)\|
    + 2x \| f_{tu}(\ellonearg) \|
    + x^2 \| f_{uu}(\ellonearg) \| \|\kone+h\kones\|^2 + 1 \big)\\
    &\leq C_K(1+ x \gamma_{\tK}(\theta^1)).
  \end{align*}

  Next, we prove an analogous statement for $\ktwo$.
  Since the arguments are essentially the same but with slightly different
  functions, we will not present all steps in as much detail but refer the
  reader to the corresponding argument for $\kone$, except for the bounds
  of $\ktwoss$.
  By definition, the vector $\ktwo(t,h,x,v)$ is the unique root of the
  function $\genvec \mapsto g_2(t,h,x,\initvec;\genvec)
  := \genvec - f(t+h (1-x), \initvec+h (1-2x)\kone+h x \genvec)$.
  The Jacobian of $g_2$ with respect to $\genvec$ is
  \begin{equation*}
    D_{\genvec}[g_2(t,h,x,\initvec;\genvec)]
    =I- h xf_u(t+h (1-x), \initvec+h (1-2x)\kone+h x \genvec).
  \end{equation*}
  By Lemma~\ref{lem:fuest}, it follows that
  $D_{\genvec}[g_2(t,h,x,\initvec;\genvec)]$ is invertible.
  So we can use \cite[Theorem~15.1]{Deimling.1985} to prove the local differentiability of
  $\ktwo$ with respect to $h$.
  Due to Lemma~\ref{lem:exist_measble} the partial derivative with respect to
  $h$ is uniquely defined on the entire interval $[0,\min(T-t,h_{\max})]$ and is
  given by
  \begin{align} \label{eq:k2s}
    \begin{split}
      \ktwos&= -\big(D_{\genvec}[g_2(t,h,x,\initvec;\ktwo)]\big)^{-1}
      D_{h}[g_2(t,h,x,\initvec;\ktwo)]\\
      &=\big(I- h xf_u(\elltwoarg)\big)^{-1} \big((1-x) f_t(\elltwoarg)
      + (1-2x) f_u(\elltwoarg) (\kone+h\kones)+x f_u(\elltwoarg) \ktwo\big).
    \end{split}
  \end{align}
  Analogously to the calculation for $\koness$, we apply
  $(V^{-1})_h(h)=-V^{-1}V_hV^{-1}(h)$ in the first step and insert
  \eqref{eq:k2s} in the second and obtain
  \begin{align}\label{eq:k2ss}
    \begin{split}
      \ktwoss
      &=-\big(I- h xf_u(\elltwoarg)\big)^{-1} \big(I- h xf_u(\elltwoarg)\big)_h
      \big(I- h xf_u(\elltwoarg)\big)^{-1} \\
      &\qquad \times \big((1-x) f_t(\elltwoarg) + (1-2x) f_u(\elltwoarg)
      (\kone+h\kones)+x f_u(\elltwoarg)  \ktwo\big)\\
      &\quad + \big(I- h xf_u(\elltwoarg)\big)^{-1} \big((1-x) f_t(\elltwoarg)
      + (1-2x) f_u(\elltwoarg)  (\kone+h\kones)+x f_u(\elltwoarg) \ktwo\big)_h\\
      &=\big(I- h xf_u(\elltwoarg)\big)^{-1} \big( xf_u(\elltwoarg)
      + h x(1-x)f_{tu}(\elltwoarg) + h xf_{uu}(\elltwoarg)[\elltwos]\big) \ktwos \\
      &\quad + \big(I- h xf_u(\elltwoarg)\big)^{-1} \Big(
      (1-x)^2 f_{tt}(\elltwoarg) + (1-x) f_{tu}(\elltwoarg)\elltwos\\
      &\qquad
      + (1-2x) \big((1-x) f_{tu}(\elltwoarg)  (\kone+h\kones)
      + f_{uu}(\elltwoarg)[\elltwos,\kone+h\kones] \\
      &\qquad + f_u(\elltwoarg)  (2\kones+h\koness)\big)
      + x (1-x) f_{tu}(\elltwoarg)  \ktwo
      + x f_{uu}(\elltwoarg)[\elltwos,\ktwo]
      + x f_u(\elltwoarg)  \ktwos \Big)\\
      &=\big(I- h xf_u(\elltwoarg)\big)^{-1} \Big( (1-x)^2 f_{tt}(\elltwoarg)
        + f_u(\elltwoarg)((1-2x) (2\kones+h\koness) + 2 x\ktwos )\\
      &\quad + (1-x)f_{tu}(\elltwoarg)(\elltwos + (1-2x) (\kone+h\kones)
      + x (\ktwo + h \ktwos) )\\
      &\quad + f_{uu}(\elltwoarg)[\elltwos,(1-2x)(\kone+h\kones)
      + x (\ktwo+ h \ktwos)]\Big) \\
      &=\big(I- h xf_u(\elltwoarg)\big)^{-1} \Big( (1-x)^2 f_{tt}(\elltwoarg)
        + 2 (1-x)f_{tu}(\elltwoarg)\elltwos + f_{uu}(\elltwoarg)[\elltwos]^2\\
      &\quad + f_u(\elltwoarg)((1-2x) (2\kones+h\koness) + 2 x\ktwos )\Big).
    \end{split}
  \end{align}
  Next, we bound the term $\ktwos$. We recall the explicit expression from
  \eqref{eq:k2s} and again apply the second
  bound from Lemma~\ref{lem:fuest} to see
  \begin{align*}
    \| \ktwos\|
    &= \big\| \big(I- h xf_u(\elltwoarg)\big)^{-1} \big((1-x) f_t(\elltwoarg)
    + (1-2x) f_u(\elltwoarg)  (f(\ellonearg)+h\kones)\\
    &\qquad +x f_u(\elltwoarg)  f(\elltwoarg)\big) \big\|\\
    &\leq (1 - h_{\max} \nu)^{-1} \big( \|f_t(\elltwoarg)\|
    + \| f_u(\elltwoarg)\| (\|f(\ellonearg)\| + h_{\max} \|\kones\| )
    + \|f_u(\elltwoarg)\| \|f(\elltwoarg)\| \big).
  \end{align*}
  Since $\kones$ is bounded with respect to a constant that depends on $K$ and
  all other appearing terms can be argued analogously as in the first part of the proof, $\ktwos$ is also
  bounded by a constant that depends on $K$.
  Inserting the definition of $\elltwos$, we find
  \begin{align*}
    \|\elltwos\| \leq |1-2x| (\|\kone\| + h\|\kones\|) + x (\|\ktwo\| + h \|\ktwos\|)
    \leq \|f(\ellonearg)\| + h_{\max} \|\kones\| + \|f(\elltwoarg)\|
    + h_{\max} \|\ktwos\|.
  \end{align*}
  By possibly again enlarging $\tK$, it therefore follows that $\ktwo$, $\ktwos$, $\elltwo$, and $\elltwos$ are bounded by this constant which depends on $K$.

  To bound $\ktwoss$ as claimed, we need a sharper bound of $\elltwos$ given by
  \begin{align*}
    \|\elltwos\| \leq (1-x)\|\kone\|+x\|\ktwo-\kone\| + h \|\kones+\ktwos\|
    \leq (1-x+h) C_K,
  \end{align*}
  where we used
  \begin{equation*}
    \|\ktwo-\kone\|= \|f(\elltwoarg)-f(\ellonearg)\|
    \leq (2x-1)h C_K + hC_K\| (1-3x)\kone + x\ktwo\| \leq h C_K,
  \end{equation*}
  which holds due to the Lipschitz continuity from~\eqref{eq:f0_Lip} in Lemma~\ref{lem:implied_continuity}.
  To show the bound for the second partial derivative of $\ktwo$, we apply the
  explicit definition from \eqref{eq:k2ss} and apply the second
  bound from Lemma~\ref{lem:fuest}. We can further argue with the continuous differentiability of $f$ and the proven
  boundedness of $\ellone$ and $\elltwo$ that $\|f_u(\elltwoarg)\|$ is bounded
  by a constant depending on $K$.
  For $\|f_{tt}(\elltwoarg)\|$, $\| f_{tu}(\elltwoarg)\|$, and
  $\|f_{uu}(\elltwoarg)\|$, we apply Assumption~\ref{ass:f2_bound} for the
  constant $\kappa = \tK$ and we use the sharper bound of $\elltwos$. This combination then implies that
  \begin{align*}
    \|\ktwoss\|
    &\leq (1 - h_{\max} \nu)^{-1} \big( (1-x)\|f_{tt}(\elltwoarg)\|
    + (1-x)\|f_{tu}(\elltwoarg)\| \|\elltwos\|
    + \|f_{uu}(\elltwoarg)\| \|\elltwos\|^2\\
    &\quad + \|f_u(\elltwoarg)\| (2\|\kones\|+h \|\koness\| + 2 \|\ktwos\| )
    \big)\\
    &\leq C_K \big( (1-x)\gamma_{\tK}(\theta^2)
    + (1-x)\gamma_{\tK}(\theta^2) + (1-x+h)^2 \gamma_{\tK}(\theta^2)
    + 1 + h x \gamma_{\tK}(\theta^1) \big)\\
    &\leq C_K \big(1 + h\gamma_{\tK}(\theta^1)
    + (1-x+h)\gamma_{\tK}(\theta^2)
    \big).
  \end{align*}
\end{proof}

Now we can rewrite $\Phi(t,h,x,\initvec)$ in such a way that the first and
second-order terms coincide with the exact flow $\Psi(\genvec,t,\tau)$ from
\eqref{eq:def_Psi}.

\begin{lemma}\label{lem:taylor_approx}
  Let Assumption~\ref{ass:f_high_conv} be fulfilled. Then the function $h
  \mapsto \Phi(t,h,x,\initvec)$ can be written as
  \begin{equation}\label{eq:Taylor_Phi}
    \begin{split}
      \Phi(t,h,x,\initvec)
      &=\initvec+ hf(t,\initvec)+ \frac{h^2}{2} \big(f_t(t,\initvec)
      +f_u(t,\initvec)f(t,\initvec) \big)\\
      &\quad +\frac{h}{2}\int_0^h (h-\tau) \big(\koness(t,\tau,x,v)
      +\ktwoss(t,\tau,x,v) \big) \diff{\tau},
    \end{split}
  \end{equation}
  for $h\in[0,\min(T-t,h_{\max})]$.
\end{lemma}

\begin{proof}
  By the definition of $\Phi$ stated in \eqref{eq:Phidef}, we find
  \begin{align*}
    \Phi(t,h,x,\initvec) = \initvec + \frac{h}{2} \big(\kone(t,h,x,\initvec) +
    \ktwo(t,h,x,\initvec) \big).
  \end{align*}
  Inserting a first-order Taylor expansion w.r.t.\ the step size $h$ of $\kone$ and $\ktwo$
  with integral remainder terms from Lemma~\ref{lem:bounds}, we obtain
  \begin{align*}
    \Phi(t,h,x,\initvec)
    &=\initvec + \frac{h}{2} \big(\kone(t,0,x,\initvec) + \ktwo(t,0,x,\initvec)
    + h \big(\kone_{h}(t,0,x,\initvec) + \ktwo_{h}(t,0,x,\initvec)\big)\big)\\
    &\quad + \frac{h}{2} \int_{0}^{h} (h -\tau) \big(\koness(t,\tau,x,\initvec)
    + \ktwoss(t,\tau,x,\initvec)\big) \diff{\tau}\\
    &=\initvec + h f(t,\initvec) + \frac{h^2}{2} \big( f_t(t,\initvec) +
    f_u(t,\initvec)f(t,\initvec) \big)\\
    &\quad + \frac{h}{2} \int_{0}^{h} (h -\tau) \big(\koness(t,\tau,x,\initvec)
    + \ktwoss(t,\tau,x,\initvec)\big) \diff{\tau}.
  \end{align*}
\end{proof}

In the following, we want to compare the remainder term of the exact flow $\Psi$
from \eqref{eq:def_Psi} to the remainder term of the numerical approximation
$\Phi$ from \eqref{eq:Phidef}.
Since our randomized RK method satisfies the third-order conditions in
expectation only, we need to take the expectation over $\koness$
and $\ktwoss$ to cancel out all lower-order terms.
First note that inserting the integral identity from Lemma~\ref{lem:calc}, we
can rewrite the second-order remainder term of $\Psi$ as follows
\begin{align*}
  &\frac{1}{2}\int_0^h (h-\tau)^2F(t+\tau, \Psi(\genvec,t,\tau))\diff{\tau}\\
  &=h\int_0^h (h-\tau)\int_0^1 x^2F(t+\tau x, \Psi(\genvec,t,\tau x))\diff{x}\diff{\tau}\\
  &=2\int_{\half}^1\frac{h}{2}\int_0^h (h-\tau) \big[x^2F(t+\tau x, \Psi(\genvec,t,\tau x))
  +(1-x)^2F(t+\tau (1-x), \Psi(\genvec,t,\tau (1-x)))\big]\diff{\tau}\diff{x}.
\end{align*}
The integral $2\int_{\half}^1\cdot\diff{x}$ is a substituted version of the
expectation that appears in Section~\ref{sec:convergence}.
Altogether, the aim in the rest of this section is to bound the term
\begin{equation}\label{eq:goal_taylor}
  x^2F(t+\tau x, \Psi(\genvec,t,\tau x)) + (1-x)^2F(t+\tau (1-x),
  \Psi(\genvec,t,\tau (1-x))) - \koness(t,\tau,x,\initvec)
  - \ktwoss(t,\tau,x,\initvec).
\end{equation}

We begin with the following result, where we express $\kone$ and
$\ktwo$ in terms of $F_1$ and $F_2$ and bound the remaining terms.

\begin{lemma}\label{lem:diff_kss_F}
  Let Assumption~\ref{ass:f_high_conv} be fulfilled.
  Then for every $K>0$, there exists a constant $C_K>0$ such that
  \begin{align*}
    \big\|\koness(t,h,x,v) - x^2 F_1(\ellonearg) - 2 x^2 F_2(\ellonearg)\big\|
    &\leq h C_K \gamma_{\tK}(\theta^1)\\
    \big\| \ktwoss(t,h,x,v) - (1-x)^2 F_1(\elltwoarg)
    - (4x-6x^2) F_2(\elltwoarg) \big\|
    &\leq h^{\frac{1}{2}} C_K + h C_K \big(1+\tL_{\tK}(\theta^1)\\
    &\hspace{3cm}+\gamma_{\tK}(\theta^1)
    + \gamma_{\tK}(\theta^2)\big)
  \end{align*}
  is fulfilled for every $t\in[0,T]$, $h\in[0,\min(T-t,h_{\max})]$,
  $x\in[\half,1]$, and $v\in\R^d$ such that $\|v\|\leq K$,
  where $\tK$, $\theta^1$, $\theta^2$, $\ellone$, $\elltwo$ are stated in
  Lemma~\ref{lem:bounds}.
\end{lemma}

\begin{proof}
  Note that we will drop the argument $(t,h,x,v)$ in the proof for better
  readability.
  Recalling the implicit form of $\kones$ and $\koness$ stated in
  Lemma~\ref{lem:bounds} and inserting the definitions of $\ellones$, we find
  \begin{align*}
    \koness
    &= x^2 f_{tt}(\ellonearg) + 2x f_{tu}(\ellonearg)\ellones
    +  f_{uu}(\ellonearg)[\ellones]^2
    + 2x f_u(\ellonearg)\kones + h x f_u(\ellonearg) \koness\\
    &= x^2 f_{tt}(\ellonearg) + 2x f_{tu}(\ellonearg) \big(xf(\ellonearg)
    + hx \kones\big) + f_{uu}(\ellonearg)[xf(\ellonearg) + hx \kones]^2\\
    &\quad+ 2x f_u(\ellonearg) \big( x f_t(\ellonearg) + x
    f_u(\ellonearg)f(\ellonearg) + h x f_u(\ellonearg) \kones \big) + h x
    f_u(\ellonearg) \koness\\
    &= x^2 f_{tt}(\ellonearg) + 2x^2 f_{tu}(\ellonearg)f(\ellonearg)
    + x^2 f_{uu}(\ellonearg)[f(\ellonearg)]^2\\
    &\quad+ 2x^2 f_u(\ellonearg) f_t(\ellonearg) + 2x^2 f_u^2(\ellonearg)
    f(\ellonearg)\\
    &\quad + 2h x^2 f_{tu}(\ellonearg) \kones
    +  2 h x^2 f_{uu}(\ellonearg)[f(\ellonearg), \kones]
    +  h^2 x^2 f_{uu}(\ellonearg)[\kones]^2\\
    &\quad+ 2h x^2  f_u^2(\ellonearg) \kones + h x f_u(\ellonearg) \koness\\
    &= x^2 F_1(\ellonearg)+2 x^2 F_2(\ellonearg)+ 2h x^2 f_{tu}(\ellonearg)
    \kones +  2 h x^2 f_{uu}(\ellonearg)[f(\ellonearg), \kones]\\
    &\quad +  h^2 x^2 f_{uu}(\ellonearg)[\kones]^2 + 2h x^2  f_u^2(\ellonearg)
    \kones + h x f_u(\ellonearg) \koness.
  \end{align*}
  Note that due to Lemma~\ref{lem:bounds} the argument $\ellone$ is bounded by a
  constant $\tK$ depending on $K$. Thus, inserting Assumption~\ref{ass:f2_bound} for
  $\kappa = \tK$, shows that
  \begin{align*}
    &\big\| \koness - x^2 F_1(\ellonearg) -2 x^2 F_2(\ellonearg) \big\|\\
    &\leq h x \big( 2x \| f_{tu}(\ellonearg)\| \|\kones\|
    +  2 x \| f_{uu}(\ellonearg)\| \|f(\ellonearg)\| \|\kones\|
    +  h x \| f_{uu}(\ellonearg)\| \|\kones\|^2 \\
    &\quad+ 2x \| f_u(\ellonearg)\|^2 \|\kones\| + \| f_u(\ellonearg)\|
    \| \koness\| \big)\\
    &\leq h C_K \gamma_{\tK}(\theta^1).
  \end{align*}
  For the term $\ktwoss$, we can in principle argue in the same way.
  As $\ktwoss$ contains some additional terms compared with $\koness$, it
  becomes a bit more technical. We begin again by recalling the implicit form of
  $\kones$, $\ktwos$ and $\koness$ stated in Lemma~\ref{lem:bounds} and
  inserting the definitions of $\ellones$ and $\elltwos$ to find
  \begin{align*}
    \ktwoss
    & =(1- x)^2 f_{tt}(\elltwoarg) + 2(1-x)f_{tu}(\elltwoarg)\elltwos
    + f_{uu}(\elltwoarg)[\elltwos]^2
    + (1-2x) f_u(\elltwoarg) ( 2 \kones + h \koness )\\
    &\quad+ x f_u(\elltwoarg) (2 \ktwos + h \ktwoss)\\
    &=(1-x)^2 f_{tt}(\elltwoarg) + 2 (1-x)f_{tu}(\elltwoarg)\elltwos
    + f_{uu}(\elltwoarg)[\elltwos]^2\\
    &\quad + 2 (1-2x) f_u(\elltwoarg) \kones
    + 2 x f_u(\elltwoarg) \ktwos
    + h f_u(\elltwoarg) \big((1-2x) \koness + x \ktwoss \big)\\
    &=(1-x)^2 f_{tt}(\elltwoarg) + 2 (1-x)f_{tu}(\elltwoarg)
    \big((1-2x)f(\ellonearg) + h (1-2x)\kones+ xf(\elltwoarg)+hx\ktwos\big)\\
    &\quad + f_{uu}(\elltwoarg)\big[ (1-2x)f(\ellonearg)
    + h (1-2x)\kones+ xf(\elltwoarg)+hx\ktwos \big]^2\\
    &\quad + 2 (1-2x) f_u(\elltwoarg) \big( x f_t(\ellonearg)
    + x f_u(\ellonearg) f(\ellonearg) + h x f_u(\ellonearg) \kones \big)\\
    &\quad + 2 x f_u(\elltwoarg) \big( (1-x) f_t(\elltwoarg) + f_u(\elltwoarg)
    \big((1-2x) (f(\ellonearg)+h \kones) + x (f(\elltwoarg)+h \ktwos) \big) \big) \\
    &\quad + h f_u(\elltwoarg) \big((1-2x) \koness + x \ktwoss \big)\\
    &= \Gamma_1 + h \Gamma_2,
  \end{align*}
  where we collect all the terms that are led by the factor $h$ in $\Gamma_2$
  and all the others in $\Gamma_1$.
  Apart from some mixed terms containing both $\ellone$ and $\elltwo$,
  $\Gamma_1$ is almost in the form of
  $(1-x)^2 F_1(\elltwoarg) + (4x-6x^2) F_2(\elltwoarg)$.
  Subtracting this term from $\Gamma_1$, we observe
  \begin{align*}
    &\Gamma_1 - (1-x)^2 F_1(\elltwoarg) - (4x-6x^2) F_2(\elltwoarg)\\
    &=(1-x)^2 f_{tt}(\elltwoarg) + 2 (1-x)f_{tu}(\elltwoarg)
    \big((1-2x)f(\ellonearg) + xf(\elltwoarg)\big)\\
    &\quad + f_{uu}(\elltwoarg)\big[(1-2x)f(\ellonearg) + xf(\elltwoarg)\big]^2\\
    &\quad + 2 (1-2x) xf_u(\elltwoarg) \big( f_t(\ellonearg)
    + f_u(\ellonearg) f(\ellonearg) \big)\\
    &\quad + 2 x f_u(\elltwoarg) \big( (1-x) f_t(\elltwoarg) + f_u(\elltwoarg)
    \big((1-2x) f(\ellonearg) + x f(\elltwoarg) \big) \big) \\
    &\quad - (1-x)^2 F_1(\elltwoarg) - (4x-6x^2) F_2(\elltwoarg)\\
    &= 2 (1-x) (1-2x) f_{tu}(\elltwoarg) \big(f(\ellonearg) - f(\elltwoarg)\big)\\
    &\quad+ (1-2x) f_{uu}(\elltwoarg) \big[f(\ellonearg)
    - f(\elltwoarg), (1-2x) f(\ellonearg) + x f(\elltwoarg)\big] \\
    &\quad+ (1-x) (1-2x) f_{uu}(\elltwoarg)
    \big[f(\elltwoarg), f(\ellonearg) - f(\elltwoarg) \big] \\
    &\quad + 2 (x-2x^2) f_u(\elltwoarg) \big( f_t(\ellonearg) - f_t(\elltwoarg)
    + f_u(\ellonearg) f(\ellonearg) - f_u(\elltwoarg) f(\elltwoarg)\big)\\
    &\quad + 2 (x-2x^2) f_u^2(\elltwoarg) \big( f(\ellonearg) - f(\elltwoarg)\big).
  \end{align*}
  Applying Assumption~\ref{ass:f2_bound} for $\kappa = \tK$ and
  Lemma~\ref{lem:bounds}, we can bound the norm of $\Gamma_2$ by
  \begin{align*}
    \|\Gamma_2\|
    &= \big\| 2 (1-x) f_{tu}(\elltwoarg) \big( (1-2x)\kones+ x\ktwos\big)\\
    &\quad + f_{uu}(\elltwoarg)\big[(1-2x)\kones+x\ktwos, (1-2x)f(\ellonearg)
    + h (1-2x)\kones+ xf(\elltwoarg)+hx\ktwos \big]\\
    &\quad + f_{uu}(\elltwoarg)\big[ (1-2x)f(\ellonearg)
    + xf(\elltwoarg), (1-2x)\kones+ x\ktwos \big]\\
    &\quad + 2 (1-2x) x f_u(\elltwoarg) f_u(\ellonearg) \kones
    + 2 x f_u^2(\elltwoarg) \big((1-2x) \kones + x\ktwos \big) \\
    &\quad + f_u(\elltwoarg) \big((1-2x) \koness + x \ktwoss \big)\big\|\\
    &\leq C_K \big(1 + \gamma_{\tK}(\theta^1)
    + \gamma_{\tK}(\theta^2) \big).
  \end{align*}
  Combining the previous equality and bound with Assumption~\ref{ass:f2_bound}
  for $\kappa = \tK$, we find that
  \begin{align*}
    &\big\|\ktwoss - (1-x)^2 F_1(\elltwoarg) - (4x-6x^2) F_2(\elltwoarg) \big\|\\
    &\leq C_K \big(1+\gamma_{\tK}(\theta^2)\big) \| f(\ellonearg) - f(\elltwoarg)\|
    + C_K \| f_t(\ellonearg) - f_t(\elltwoarg) \|
    + C_K \| f_u(\ellonearg)-f_u(\elltwoarg) \| \\
    &\quad + h  C_K \big(1 + \gamma_{\tK}(\theta^1) +
    \gamma_{\tK}(\theta^2) \big).
  \end{align*}
  Due to the bounds~\eqref{eq:f1_spatial} and \eqref{eq:f1_temporal} from Lemma~\ref{lem:implied_continuity}, it follows that
  \begin{align*}
    &\| f_t(\ellonearg) - f_t(\elltwoarg)\|
    + \| f_u(\ellonearg) - f_u(\elltwoarg)\|\\
    &\leq C_K |\theta^1 -\theta^2|^{\frac{1}{2}}
    + \tL_{\tK}(\theta^1) \|\ellone - \elltwo\| \\
    &\leq C_K |hx - h(1-x)|^{\frac{1}{2}}
    + \tL_{\tK}(\theta^1)\| hx \kone - h (1-2x) \kone + hx \ktwo \|\\
    &\leq h^{\frac{1}{2}} C_K |2x - 1|^{\frac{1}{2}}
    + h \tL_{\tK}(\theta^1) \| (3x-1) \kone - x \ktwo \|
    \leq h^{\frac{1}{2}} C_K+ h C_K\tL_{\tK}(\theta^1).
  \end{align*}
  Analogously with the bound~\eqref{eq:f0_Lip} from Lemma~\ref{lem:implied_continuity}, we find
  \begin{align*}
    \| f(\ellonearg) - f(\elltwoarg)\| \leq h C_K .
  \end{align*}
  Combining the previous bounds verifies the second claim of this lemma.
\end{proof}

Next, we note that the factor in front of $F_2$ in Lemma~\ref{lem:diff_kss_F} is
not the same as the one in \eqref{eq:goal_taylor}.
To bound their differences, we need the following lemma.

\begin{lemma}\label{lem:diff_kss_F_poly}
  Let Assumption~\ref{ass:f_high_conv} be fulfilled.
  Then for every $K >0$, there exists a constant $C_K>0$ such that
  \begin{align*}
    &\Big\| \int_{\frac{1}{2}}^1 \big(x^2 F_2(\ellonearg)
    + (1-x)^2 F_2(\elltwoarg) - 2 x^2 F_2(\ellonearg)
    - (4x-6x^2) F_2(\elltwoarg)\big) \diff{x} \Big\|
    \leq h^{\frac{1}{2}} C_K,
  \end{align*}
  is fulfilled for every $t\in[0,T]$, $h\in[0,\min(T-t,h_{\max})]$,
  $v\in\R^d$ such that $\|v\|\leq K$, where $\theta^1$, $\theta^2$, $\ellone$, $\elltwo$ are stated in
  Lemma~\ref{lem:bounds}.
\end{lemma}

\begin{proof}
  To bound the integral, we add and subtract $F_2(t,v)$ in suitable ways.
  Since $\theta^i$ and $\ell^i$ depend on $x$ for $i \in \{1,2\}$, we cannot apply arguments on the polynomial in $x$ in a straightforward way. Thus, exchanging $F_2(\ellonearg)$ and $F_2(\elltwoarg)$ by $F_2(t,v)$, simplifies some arguments.
  We then see that
  \begin{align*}
    &\Big\| \int_{\frac{1}{2}}^1 \big(x^2 F_2(\ellonearg) + (1-x)^2 F_2(\elltwoarg)
    - 2 x^2 F_2(\ellonearg) - (4x-6x^2) F_2(\elltwoarg)\big) \diff{x} \Big\| \\
    &\quad =\Big\| \int_{\frac{1}{2}}^1 \big(- x^2 F_2(\ellonearg)
    + (1-6x+7x^2) F_2(\elltwoarg) \big) \diff{x} \Big\|\\
    &\quad \leq \Big\| \int_{\frac{1}{2}}^1 x^2 \big(F_2(t,v) - F_2(\ellonearg)\big) \diff{x}
    +  \int_{\frac{1}{2}}^1 \big(- x^2 + 1-6x+7x^2\big) \diff{x} F_2(t,v)\\
    &\qquad + \int_{\frac{1}{2}}^1 (1-6x+7x^2)
    \big( F_2(\elltwoarg) - F_2(t,v)\big)\diff{x}\Big\| \\
    &\quad \leq \int_{\frac{1}{2}}^1 x^2 \big\|F_2(t,v) - F_2(\ellonearg)\big\| \diff{x}
    + 0 + \int_{\frac{1}{2}}^1 |1-6x+7x^2|
    \big\| F_2(\elltwoarg) - F_2(t,v)\big\| \diff{x}\\
    &\quad \leq C_K \int_{\frac{1}{2}}^1 x^2 \big(|t -\theta^1|^{\frac{1}{2}}
    + \tL_{\widehat{K}}(\theta^1) \| v - \ellone \| \big)\diff{x}
    + C \int_{\frac{1}{2}}^1  \big(|t -\theta^2|^{\frac{1}{2}}
    + \tL_{\widehat{K}}(\theta^2) \| v - \elltwo \|\big) \diff{x} \\
    &\quad \leq C_K \int_{\frac{1}{2}}^1 \big(h^{\frac{1}{2}}
    + \tL_{\widehat{K}}(\theta^1)\| hx\kone(t,h,x,\initvec) \|+ \tL_{\widehat{K}}(\theta^2)\| h(1-2x)\kone(t,h,x,v)
    +h x \ktwo(t,h,x,v) \| \big)\diff{x}\\
    &\quad \leq C_K \Big(h^{\half} + h\int_{\frac{1}{2}}^1 \big(
    \tL_{\widehat{K}}(\theta^1) + \tL_{\widehat{K}}(\theta^2) \big)\diff{x}\Big)
    = C_K \Big(h^{\half} + \int_{t}^{t + h}
    \tL_{\widehat{K}}(x) \diff{x}\Big)\\
    &\quad \leq C_K \Big(h^{\half} + h^{\frac{1}{2}} \Big(\int_{t}^{t + h}
    \tL_{\widehat{K}}^2(x) \diff{x}\Big)^{\frac{1}{2}} \Big)
    = h^{\frac{1}{2}} C_K,
  \end{align*}
  where we applied Lemma~\ref{lem:lipschitzF} and Lemma~\ref{lem:bounds},
  made use of the fact that $\int_{\half}^1(1-6x+6x^2)\diff{x}=0$, H\"older's inequality and used the substitutions $t + hx \to x$ and $t + h(1-x) \to x$.
\end{proof}

The last step for comparing the Taylor approximations of $\Phi$ and $\Psi$, is to
exchange the arguments $\ellone$ and $\elltwo$ to the arguments
$\Psi(\genvec,t,hx)$ and $\Psi(\genvec,t,h(1-x))$, respectively.
For this, we use the following result about the difference between $\ellone$ and
$\Psi(\genvec,t,hx)$ and the difference between $\elltwo$ and
$\Psi(\genvec,t,h(1-x))$.

\begin{lemma}\label{lem:ell_approx}
  Let Assumption~\ref{ass:f_high_conv} be fulfilled.
  Then for every $K>0$, there exists a constant $C_K>0$ such that
  for all $t\in[0,T]$, $h\in[0,\min(T-t,h_{\max})]$,
  $x\in[\half,1]$ and $v \in \R^d$ with $\|v\|\leq K$ it holds
  \begin{align*}
    \| \ellone - \Psi(v,t,hx)\| \leq h^{2} C_K
    \quad \text{and} \quad
    \| \elltwo - \Psi(v,t,h(1-x))\| \leq h^{2} C_K,
  \end{align*}
  where
  \begin{align*}
    \ellone:= v +hx\kone(t,h,x,v), \qquad
    \elltwo:= v +h(1-2x)\kone(t,h,x,v) + h x \ktwo(t,h,x,v).
  \end{align*}
\end{lemma}

\begin{proof}
  We begin to show some first-order estimates that we can then later improve with
  a bootstrap argument.
  After inserting the definitions of $\ellone$ and $\elltwo$, we observe that
  \begin{align*}
    \|\ellone - \elltwo\|
    &= h \| x\kone(t,h,x,v) - (1-2x)\kone(t,h,x,v) - x \ktwo(t,h,x,v) \|
    \leq h C_K
  \end{align*}
  due to Lemma~\ref{lem:bounds}.
  Similarly with the definition of $\Psi$ from \eqref{eq:def_Psi} and its bound
  \eqref{eq:boundPsi}, we can apply the same lemma as well as
  the bound~\eqref{eq:f0_Lip} to obtain
  \begin{align*}
    \|\ellone - \Psi(v,t,s)\|
    &= \Big\| h x \kone(t,h,x,v) - \int_{0}^{s} f(t+ \tau, \Psi(v,t,\tau))
    \diff{\tau} \Big\|
    \leq h C_K
  \end{align*}
  for any $s\in[0,h]$.
  Applying the triangle inequality, we analogously find that
  $\|\elltwo - \Psi(v,t,s)\| \leq h C_K$.

  After these preparations, we show the first claim of the lemma.
  Inserting the Lipschitz condition on $f$ from the bound~\eqref{eq:f0_Lip} in Lemma~\ref{lem:implied_continuity} and the first order bound of the difference of $\ellone$ and $\Psi(v,t,hx)$, it
  follows that
  \begin{align*}
    \| \ellone - \Psi(v,t,hx) \|
    &\leq \int_{0}^{hx} \big\| f(t+hx, \ellone)
    - f(t+ \tau, \Psi(v,t,\tau))\big\| \diff{\tau}\\
    &\leq C_K \int_{0}^{hx} \big(| hx-\tau|
    + \| \ellone - \Psi(v,t,\tau) \|\big) \diff{\tau}
    \leq h^2 C_K.
  \end{align*}
  Analogously, we can argue for $\elltwo$ that
  \begin{align*}
    &\| \elltwo - \Psi(v,t,h(1-x)) \| \\
    &= \Big\| h(1-2x) f(t+hx, \ellone) + hx f\big(t+h (1-x), \elltwo\big)
    - \int_{0}^{h(1-x)} f(t+ \tau, \Psi(v,t,\tau)) \diff{\tau} \Big\| \\
    &\leq h(1-2x) \big\|f(t+hx, \ellone) - f\big(t+h (1-x), \elltwo\big)\big\|\\
    &\quad + \int_{0}^{h(1-x)} \big\| f\big(t+h (1-x), \elltwo\big)
    - f(t+ \tau, \Psi(v,t,\tau)) \big\| \diff{\tau} \\
    &\leq h C_K \big(h |2x-1| + \|\ellone - \elltwo\| \big)
    + C_K \int_{0}^{h(1-x)} \big(|h (1-x) - \tau|
    + \big\| \elltwo-\Psi(v,t,\tau) \big\| \big) \diff{\tau}
    \leq h^2 C_K.
  \end{align*}
\end{proof}

Now we combine the previous lemmas to get the following result.

\begin{lemma}\label{lem:conc_taylor}
  Let Assumption~\ref{ass:f_high_conv} be fulfilled.
  Then for every $K>0$, there exist a constant $C_K>0$ such that
  \begin{align*}
    &\Big\|\int_{\half}^1 \big(x^2F(\theta^1, v(\theta^1))
    + (1-x)^2F(\theta^2, v(\theta^2)) - \koness(t,h,x,\initvec)
    - \ktwoss(t,h,x,\initvec)\big)\diff{x} \Big\|\\
    &\leq  h^{\half} C_K + h C_K \Big(1 + \int_{\half}^1 \big(
    \tL_{\widehat{K}}(\theta^1)+\gamma_{\widehat{K}}(\theta^1) + \gamma_{\widehat{K}}(\theta^2)\big)\diff{x}\Big)\\
    &\quad+ h^{2\sigma} C_K \int_{\half}^1 \big(x^2(L_{\widehat{K}}(\theta^1)
    + \gamma_{\widehat{K}}(\theta^1))+(1-x)^2
    (L_{\widehat{K}}(\theta^2)+\gamma_{\widehat{K}}(\theta^2))\big)\diff{x}.
  \end{align*}
  with the abbreviations
  \begin{align*}
    &\theta^1:=t+hx ,&& \theta^2:= t+h(1-x),\\
    &v(\theta^1):=\Psi(v,t,hx) ,&& v(\theta^2):= \Psi(v,t,h (1-x)),\\
    &\ellone:= v +hx\kone(t,h,x,v),
    && \elltwo:= v +h(1-2x)\kone(t,h,x,v) +h x \ktwo(t,h,x,v),\\
    & \max(\|v(\theta^1)\|, \|\ellone\|, \|v(\theta^2)\|,
    \|\elltwo\|) \leq \widehat{K} &&\text{for all } x \in [\tfrac{1}{2},1] \text{ and } h \leq h_{\mathrm{max}},
  \end{align*}
  is fulfilled for every $t\in[0,T]$, $n\in \{1,\dots,N\}$,
  $h\in[0,\min(T-t,h_{\max})]$, and $v\in\R^d$ with $\|v\|\leq K$.
\end{lemma}

\begin{proof}
  We begin by splitting up the term into the following four parts
  \begin{align*}
    &\Big\|\int_{\half}^1 \big(x^2F(\theta^1, v(\theta^1))
    + (1-x)^2F(\theta^2, v(\theta^2)) - \koness(t,h,x,\initvec)
    - \ktwoss(t,h,x,\initvec)\big)\diff{x} \Big\|\\
    &\leq \Big\|\int_{\half}^1 \big(x^2F(\theta^1, v(\theta^1))
    + (1-x)^2F(\theta^2, v(\theta^2)) - x^2 F(\theta^1, \ellone)
    - (1-x)^2F(\theta^2, \elltwo)\big)\diff{x}\Big\|\\
    &\quad + \Big\| \int_{\half}^1 \big(x^2F_2(\theta^1, \ellone)
    + (1-x)^2F_2(\theta^2, \elltwo) - 2x^2F_2(\theta^1, \ellone)
    - (4x-6x^2)F_2(\theta^2, \elltwo)\big)\diff{x} \Big\|\\
    &\quad + \Big\| \int_{\half}^1 \big( x^2F_1(\theta^1, \ellone)
    + 2x^2F_2(\theta^1, \ellone) - \koness(t,h,x,\initvec) \big)\diff{x}\Big\|\\
    &\quad + \Big\| \int_{\half}^1 \big( (1-x)^2F_1(\theta^2, \elltwo)
    + (4x-6x^2)F_2(\theta^2, \elltwo) - \ktwoss(t,h,x,\initvec) \big)\diff{x} \Big\|\\
    &= \Gamma_1 + \Gamma_2 + \Gamma_3 + \Gamma_4.
  \end{align*}
  For $\Gamma_1$, we apply the H\"older and Lipschitz continuity result from
  Lemma~\ref{lem:lipschitzF} and the fact that the arguments can be bounded as
  shown in Lemma~\ref{lem:ell_approx}. It then follows that
  \begin{align*}
    \Gamma_1
    &= \Big\|\int_{\half}^1 \big(x^2F(\theta^1, v(\theta^1))
    + (1-x)^2F(\theta^2, v(\theta^2)) - x^2F(\theta^1, \ellone)
    - (1-x)^2F(\theta^2, \elltwo)\big)\diff{x}\Big\| \\
    &\leq \int_{\half}^1 x^2 \big\| F_1(\theta^1, v(\theta^1))
    - F_1(\theta^1, \ellone)\big\| \diff{x}
    + \int_{\half}^1 (1-x)^2 \big\| F_1(\theta^2, v(\theta^2))
    - F_1(\theta^2, \elltwo)\big\| \diff{x}\\
    &\quad + \int_{\half}^1 x^2 \big\| F_2(\theta^1, v(\theta^1))
    - F_2(\theta^1, \ellone)\big\| \diff{x}
    + \int_{\half}^1 (1-x)^2 \big\| F_2(\theta^2, v(\theta^2))
    - F_2(\theta^2, \elltwo)\big\| \diff{x}\\
    &\leq C_K \int_{\half}^1 x^2 \big(L_{\widehat{K}}(\theta^1)
    \| v(\theta^1) - \ellone \|^{\sigma}
    + \gamma_{\widehat{K}}(\theta^1) \| v(\theta^1) - \ellone \| \big)\diff{x}\\
    &\quad + C_K \int_{\half}^1 (1-x)^2 \big(L_{\widehat{K}}(\theta^2)
    \| v(\theta^2) - \elltwo \|^{\sigma} + \gamma_{\widehat{K}}(\theta^2)
    \| v(\theta^2) - \elltwo \| \big)\diff{x} \\
    &\quad + C_K \int_{\half}^1 \big(x^2 \|v(\theta^1) - \ellone\|^{\frac{1}{2}}
    + (1-x)^2 \| v(\theta^2) - \elltwo \|^{\frac{1}{2}} \big) \diff{x}\\
    &\leq h C_K + h^{2\sigma} C_K \int_{\half}^1
    \big(x^2(L_{\widehat{K}}(\theta^1) + \gamma_{\widehat{K}}(\theta^1))
    +(1-x)^2(L_{\widehat{K}}(\theta^2)+\gamma_{\widehat{K}}(\theta^2))\big)\diff{x},
  \end{align*}
  where $\widehat{K} \geq \max(\|v(\theta^1)\|, \|\ellone\|, \|v(\theta^2)\|,
  \|\elltwo\| )$.
  The second term $\Gamma_2$ can then be bounded with
  Lemma~\ref{lem:diff_kss_F_poly}. We obtain that
  \begin{equation*}
    \Gamma_2
    =\Big\| \int_{\half}^1 \big(x^2F_2(\theta^1, \ellone)
    + (1-x)^2F_2(\theta^2,\elltwo) - 2x^2F_2(\theta^1, \ellone)
    - (4x-6x^2)F_2(\theta^2, \elltwo)\big)\diff{x} \Big\|
    \leq h^{\frac{1}{2}} C_K.
  \end{equation*}
  For $\Gamma_3$ and $\Gamma_4$, we can directly apply
  Lemma~\ref{lem:diff_kss_F}, to find
  \begin{align*}
    \Gamma_3
    = \Big\| \int_{\half}^1 \big( x^2F_1(\theta^1, \ellone)
    + 2x^2 F_2(\theta^1, \ellone) - \koness(t,h,x,\initvec) \big)\diff{x} \Big\|
    \leq h C_K \int_{\half}^1 \gamma_{\widehat{K}}(\theta^1)\diff{x}
  \end{align*}
  and
  \begin{align*}
    \Gamma_4 &= \Big\| \int_{\half}^1 \big( (1-x)^2F_1(\theta^2, \elltwo) +
    (4x-6x^2)F_2(\theta^2, \elltwo) - \ktwoss(t,h,x,\initvec) \big) \diff{x} \Big\|\\
    &\leq h^{\frac{1}{2}} C_K + h C_K \Big(1 + \int_{\half}^1
    \big(\tL_{\tK}(\theta^1)+\gamma_{\widehat{K}}(\theta^1)
    + \gamma_{\widehat{K}}(\theta^2)\big) \diff{x}\Big).
  \end{align*}
  Thus, combining the bounds for $\Gamma_1$--$\Gamma_4$, we obtain the claimed
  result
  \begin{align*}
    &\Big\|\int_{\half}^1 \big(x^2F(\theta^1, v(\theta^1))
    + (1-x)^2F(\theta^2, v(\theta^2)) - \koness(t,h,x,\initvec)
    - \ktwoss(t,h,x,\initvec)\big)\diff{x} \Big\|\\
    &\quad \leq  h^{\half} C_K + h C_K \Big(1 + \int_{\half}^1
    \big(\tL_{\tK}(\theta^1)+\gamma_{\widehat{K}}(\theta^1)
    + \gamma_{\widehat{K}}(\theta^2)\big) \diff{x}\Big)\\
    &\qquad+ h^{2\sigma} C_K \int_{\half}^1 \big(x^2(L_{\widehat{K}}(\theta^1)
    + \gamma_{\widehat{K}}(\theta^1))+(1-x)^2(L_{\widehat{K}}(\theta^2)
    +\gamma_{\widehat{K}}(\theta^2))\big)\diff{x}.
  \end{align*}
\end{proof}

\section{Convergence of the randomized Runge--Kutta method}
\label{sec:convergence}
In this section, we provide the main theoretical result of this work.
We prove an error bound for the randomized RK method \eqref{scheme:SDIRK}.
More precisely, we show that the global error denoted by
\begin{equation}\label{eq:def_global_error}
  e_{n}=u(t_{n}) - u^{n}, \quad n \in \{1,\dots,N\},
\end{equation}
converges to zero in the limit $h \to 0$ with respect to the root-mean-square
norm and determine the order of convergence.

First note that the method is B-convergent of order $1$. This will later help to
avoid having the Lipschitz constant of $f$ inside of an exponential.

\begin{lemma}[B-convergence]\label{lem:B_conv}
	Let Assumptions~\ref{ass:f_high_conv} and \ref{ass:rand_var} be fulfilled.
	For scheme \eqref{scheme:SDIRK}, there exists a constant $C_3 >0$ that only depends on the Butcher tableau, $T$, $\nu$ and $h_{max}$ such that
	\begin{align*}
		\E\big[\|e_{n+1}\|^2\big]
		\leq h^2 C_3 \max_{t\in[0,T]}\|u_{tt}(t)\|^2
	\end{align*}
	for all $n \in \{0,1,\dots,N-1\}$ and $h = \frac{T}{N}$ with $h \in (0,h_{\max}]$.
\end{lemma}

\begin{proof}
	The proof is almost the same as \cite[Theorem IV.15.3]{Hairer1996} with $q=1$, where one
	only needs to verify that the error constant can be uniformly bounded w.r.t.
	$\xi_n$.
\end{proof}

For the error analysis, we first derive estimates for the global error in terms of the local error
\begin{equation}\label{eq:def_local_error}
  \varrho_{n+1}=\Psi(u^n,t_n,h) - u^{n+1},
  \quad n = \{0,\dots,N-1\},
\end{equation}
where $\Psi$ is the exact flow of the ODE defined in \eqref{eq:def_Psi}.

\begin{lemma}\label{lem:error_bound}
	Let Assumptions~\ref{ass:f_high_conv} and \ref{ass:rand_var} be fulfilled.
	For scheme \eqref{scheme:SDIRK}, there exists a constant $C_2 >0$ that only depends on $\nu$ and $h_{\max}$ such that
	\begin{align*}
		\E\big[\|e_{n+1}\|^2\big]
		\leq \big( 1 + h C_2 + h^2 C \mathrm{e}^{2 \nuplus h} \big)
		\E \big[\|e_n\|^2 \big] + 2 \E \big[\|\varrho_{n+1}\|^2\big]
		+ h^{-1} \E\big[\big\|\E\big[\varrho_{n+1}\vert \mathcal{F}_n\big]\big\|^2 \big]
	\end{align*}
	for all $n = 0,\dots,N-1$ and $h = \frac{T}{N}$ with $h \in (0,h_{\max}]$.
	Moreover, the constant $C >0$ depends on $M_{\mathrm{exact}}$ and
  $M_{\mathrm{apriori}}$, which are stated in Lemma~\ref{lem:uexist}
  and Lemma~\ref{lem:apriori_bound}.
\end{lemma}

\begin{proof}
	Using the fact that $\|a\|^2= \|a- b\|^2 + 2\inner{a}{b} - \|b\|^2$, we obtain
\begin{align*}
	\|e_{n+1}\|^2 - \|e_{n}\|^2
	&= \|e_{n+1} - \varrho_{n+1} \|^2 + 2
	\inner{e_{n+1}}{\varrho_{n+1}} -
	\|\varrho_{n+1}\|^2 - \|e_{n}\|^2\\
	&= \big(\|e_{n+1} - \varrho_{n+1} \|^2  - \|e_{n}\|^2 \big)
	+ \big(2 \inner{e_{n+1} - e_n}{\varrho_{n+1}} -
	\|\varrho_{n+1}\|^2\big)
	+ 2 \inner{e_{n}}{\varrho_{n+1}}\\
	&=: \Gamma_1 + \Gamma_2 + \Gamma_3.
\end{align*}
We bound $\Gamma_1$ by using \cite[Lemma~IV.12.1]{Hairer1996}
\begin{equation*}
	\Gamma_1
	=\|e_{n+1}-\varrho_{n+1}\|^2-\|e_n\|^2
	=\|u(t_{n+1})-\Psi(u^n,t_n,h)\|^2 - \|e_{n}\|^2
	\leq  (\mathrm{e}^{2\nuplus h}-1)\|e_{n}\|^2.
\end{equation*}
For $\Gamma_2$, we have
\begin{align*}
	\E[\Gamma_2]
  &= 2 \E\big[
  \inner{e_{n+1} - e_n}{\varrho_{n+1}}\big]
  - \E \big[ \|\varrho_{n+1}\|^2\big]\\
	&= 2\E \big[\inner{u(t_{n+1})- u^{n+1} - u(t_n) +u^n
  +\Psi(u^n,t_n,h) -\Psi(u^n,t_n,h)}{\varrho_{n+1}}\big]
  - \E \big[ \|\varrho_{n+1}\|^2\big]\\
  &= 2\E \big[\inner{u(t_{n+1})-u(t_n)-\Psi(u^n,t_n,h)+u^n +\varrho_{n+1}}{\varrho_{n+1}}\big]
  - \E \big[ \|\varrho_{n+1}\|^2\big]\\
  &= 2\E\big[\inner{u(t_{n+1})-u(t_n)-\Psi(u^n,t_n,h)+u^n}{\varrho_{n+1}}\big]
  + \E \big[ \|\varrho_{n+1}\|^2\big]\\
	&\leq \E \big[\|u(t_{n+1})-u(t_n)-\Psi(u^n,t_n,h)+u^n \|^2 \big]
	+ 2 \E \big[\|\varrho_{n+1}\|^2\big]\\
	&= \Gamma_{2,1} + 2 \E \big[\|\varrho_{n+1}\|^2\big].
\end{align*}
To find a bound for $\Gamma_{2,1}$, we insert the ODE, apply H\"older's
inequality, the bound \eqref{eq:f0_Lip} from Lemma~\ref{lem:implied_continuity}, and finally
\cite[IV.Lemma~12.1]{Hairer1996} to find
\begin{align*}
	\Gamma_{2,1}
	&= \E \Big[\Big\|\int_0^h f(t_n+\tau,u(t_n+s))
	-f(t_n+\tau,\Psi(u^n,t_n,\tau))\diff{\tau} \Big\|^2 \Big] \\
	&\leq h \int_0^h \E \big[ \|f(t_n+\tau,u(t_n+\tau))
	-f(t_n+\tau,\Psi(u^n,t_n,\tau))\|^2  \big]\diff{\tau} \\
	&\leq h C \int_0^h \E \big[ \|u(t_n+\tau)-\Psi(u^n,t_n,\tau)\|^2\big]
	\diff{\tau} \\
	&\leq h C \int_0^h \mathrm{e}^{2 \nu \tau} \E \big[\|e_n\|^2 \big] \diff{\tau}
	\leq h^2 C \mathrm{e}^{2 \nuplus h} \E \big[\|e_n\|^2 \big].
\end{align*}
for a constant $C >0$ depending on $M_{\mathrm{apriori}}$ and $M_{\mathrm{exact}}$,
which are stated in Lemma~\ref{lem:apriori_bound} and Lemma~\ref{lem:uexist}.
Thus, it follows that
\begin{align*}
	\E[\Gamma_2]
	\leq h^2 C \mathrm{e}^{2 \nuplus h} \E \big[\|e_n\|^2 \big]
	+ 2 \E \big[\|\varrho_{n+1}\|^2\big].
\end{align*}
Applying the $L^2$-projection of the conditional expectation with $\mathcal{F}_n$ and that $e_n$ is $\mathcal{F}_n$-measurable, we obtain
\begin{align*}
	\E\big[\Gamma_3\big]
	= 2\E\big[\inner{e_n}{\varrho_{n+1}}\big]
	= 2\E \big[\inner{e_n}{\E\big[\varrho_{n+1}\vert \mathcal{F}_n\big]}\big]
	\leq h\E\big[\|e_n\|^2\big]+ h^{-1} \E\big[\big\|\E\big[\varrho_{n+1}\vert
  \mathcal{F}_n\big]\big\|^2 \big],
\end{align*}
where we used H\"older's inequality in the last step.
We can now combine the bounds for $\Gamma_1$, $\Gamma_2$ and $\Gamma_3$ to
obtain
\begin{align*}
	\E\big[\|e_{n+1}\|^2\big]
	&\leq \big( \mathrm{e}^{2\nuplus h}-1 + h + h^2 C \mathrm{e}^{2 \nuplus h}\big) \E \big[\|e_n\|^2 \big]
	+ 2 \E \big[\|\varrho_{n+1}\|^2\big]+ h^{-1} \E\big[\big\|\E\big[\varrho_{n+1}\vert
	\mathcal{F}_n\big]\big\|^2 \big]\\
	&=: \big( 1 + h C_2 + h^2 C \mathrm{e}^{2 \nuplus h} \big)
  \E \big[\|e_n\|^2 \big] + 2 \E \big[\|\varrho_{n+1}\|^2\big]
	+ h^{-1} \E\big[\big\|\E\big[\varrho_{n+1}\vert \mathcal{F}_n\big]\big\|^2 \big].
\end{align*}
\end{proof}

As a next step, we look at the expectation of the square of the norm of the
local error $\varrho_{n+1}$.

\begin{lemma}\label{lem:E_norm_loc_cond}
  Let Assumptions~\ref{ass:f_high_conv} and \ref{ass:rand_var} be fulfilled and let
  $\varrho_{n+1}$ be given as in \eqref{eq:def_local_error}.
  For scheme \eqref{scheme:SDIRK}, it then follows that there exist constants
  $C, \widehat{K} >0$ that depend on $M_{\mathrm{apriori}}$, which is stated in
  Lemma~\ref{lem:apriori_bound}, such that
  \begin{equation*}
    \E \big[\|\varrho_{n+1}\|^2\big]
    \leq h^6 C + h^5 C \int_{t_n}^{t_{n+1}} \gamma^2_{\widehat{K}}(\tau) \diff{\tau},
  \end{equation*}
  for all $n = 0,1,\dots,N-1$ and $h = \frac{T}{N}$ with $h \in (0,h_{\max}]$.
\end{lemma}

\begin{proof}
  To bound $\varrho_{n+1}$, we first look at the Taylor
  approximations of $\Psi(u^n,t_n,h)$ and $u^{n+1}$ that we have constructed
  earlier. We recall the approximation of $\Psi(u^n,t_n,h)$ as stated in
  Equations~\eqref{eq:Taylor_Psi} and find
  \begin{align*}
    \Psi(u^n,t_n,h)
    &=u^n+ hf(t_n,u^n)+ \frac{h^2}{2} (f_t(t_n,u^n)+f_u(t_n,u^n)f(t_n,u^n))\\
    &\quad +\frac{1}{2}\int_0^h (h-\tau)^2F(t_n+\tau, \Psi(u^n,t_n,\tau))\diff{\tau},
  \end{align*}
  where $F$ is given in~\eqref{eq:def_F}. Since $u^{n+1}$ can be rewritten to
  $\Phi(t_n,h,\xi_{n+1},u^n)$ by the definition of the method, we may apply
  Lemma~\ref{lem:taylor_approx} to obtain similarly
  \begin{align*}
    u^{n+1}
    &=u^n+ hf(t_n,u^n)+ \frac{h^2}{2} \big(f_t(t_n,u^n)+f_u(t_n,u^n)f(t_n,u^n) \big)\\
    &\quad +\frac{h}{2}\int_0^h (h-\tau) \big(
    \koness(t_n,\tau,\xi_{n+1},u^n)+\ktwoss(t_n,\tau,\xi_{n+1},u^n)
    \big)\diff{\tau}.
  \end{align*}
  Thus, the norm squared of $\varrho_{n+1}$ takes on the form
  \begin{align*}
    \|\varrho_{n+1}\|^2
    &= \| \Psi(u^n,t_n,h) - \Phi(t_n,h,\xi_{n+1},u^n) \|^2 \\
    &\leq \frac{1}{2} \Big\| \int_{0}^{h} (h-\tau)^2
    F(t_n + \tau, \Psi(u^n,t_n,\tau) )\diff{\tau} \Big\|^2\\
    &\quad + \frac{h^2}{2} \Big\| \int_0^h (h-\tau)
    \big(\koness(t_n,\tau,\xi_{n+1},u^n)
    +\ktwoss(t_n,\tau,\xi_{n+1},u^n) \big)\diff{\tau} \Big\|^2
    = \Gamma_1 + \Gamma_2.
  \end{align*}
  In the following, we will look at $\Gamma_1$ and $\Gamma_2$ separately.
  For $\Gamma_1$, we apply H\"{o}lder's inequality and
  Lemma~\ref{lem:lipschitzF} to obtain
  \begin{align*}
    \Gamma_1
    &= \frac{1}{2} \Big\| \int_{0}^{h} (h-\tau)^2
    F(t_n + \tau, \Psi(u^n,t_n,\tau) )\diff{\tau} \Big\|^2\\
    &\leq \frac{1}{2} \int_{0}^{h} (h-\tau)^4 \diff{\tau} \int_{0}^{h}
    \| F(t_n + \tau, \Psi(u^n,t_n,\tau) )\|^2 \diff{\tau} \\
    &= \frac{h^5 }{10} \int_{0}^{h} \| F(t_n + \tau, \Psi(u^n,t_n,\tau) )\|^2 \diff{\tau}
    \leq h^5 C \int_{t_n}^{t_{n+1}} \gamma_{\widehat{K}}^2(\tau) \diff{\tau},
  \end{align*}
  where $\widehat{K}$ is chosen such $ \| \Psi(u^n,t_n,\tau)\| \leq \widehat{K}$ for
  every $n\in\{0,\dots,N-1\}$.
  Note that such a $\widehat{K}$ exists because of Lemma~\ref{lem:apriori_bound}
  and Equation~\eqref{eq:boundPsi}.
  For $\Gamma_2$, we first take the expectation and apply Lemma~\ref{lem:bounds}
  and substitute $x \to 1 - x$ in one integral such that
  \begin{align*}
    \E\big[\Gamma_2\big]
    &= \frac{h^2}{2} \E\Big[ \Big\| \int_0^h (h-\tau)
    \big(\koness(t_n,\tau,\xi_{n+1},u^n) + \ktwoss(t_n,\tau,\xi_{n+1},u^n) \big)
    \diff{\tau} \Big\|^2 \Big]\\
    &\leq 2 \frac{h^2}{2} \int_{0}^{h} (h -\tau)^2 \diff{\tau} \int_{0}^{h}
    \int_{\half}^1 \big\|\koness(t_n,\tau,x,u^n) + \ktwoss(t_n,\tau,x,u^n)
    \big\|^2\diff{x} \diff{\tau}\\
    &\leq h^5 C \int_{0}^{h} \Big(\int_{\half}^1
    \big(1+ (x^2+\tau^2)\gamma^2_{\widehat{K}}(t_n+\tau x) \big) \diff{x}
    + \int_{\half}^1 (1-x + \tau)^2 \gamma^2_{\widehat{K}}(t_n+\tau (1-x))\diff{x}
    \Big) \diff{\tau}\\
    &= h^5 C \int_{0}^{h} \Big(\int_{\half}^1
    \big(1+ (x^2+\tau^2)\gamma^2_{\widehat{K}}(t_n+\tau x) \big) \diff{x}
    + \int_{0}^{\half} (x + \tau)^2 \gamma^2_{\widehat{K}}(t_n+\tau x)\diff{x}
    \Big) \diff{\tau}\\
    &\leq h^6 C + h^5 C \Big(\int_{0}^1 \int_{0}^{h} x^2
    \gamma^2_{\widehat{K}}(t_n+\tau x) \diff{\tau} \diff{x}
    + \int_{0}^{h} \int_{0}^1 \tau^2 \gamma^2_{\widehat{K}}(t_n+\tau x)
    \diff{x} \diff{\tau}\Big)
  \end{align*}
  for a constant $C$ that depends on $M_{\mathrm{apriori}}$ from
  Lemma~\ref{lem:apriori_bound}.
  In the two inner integrals, we now apply the substitutions $\tau x \to \tau$ and
  $\tau x \to x$ and can then bound the expectation of $\Gamma_2$ as follows
  \begin{align*}
    \E\big[\Gamma_2\big]
    &\leq h^6 C + h^5 C \Big(\int_{0}^1 \int_{0}^{h x} x
    \gamma^2_{\widehat{K}}(t_n+\tau) \diff{\tau} \diff{x}
    + \int_{0}^{h} \int_{0}^{\tau} \tau \gamma^2_{\widehat{K}}(t_n+x) \diff{x}
    \diff{\tau}\Big)\\
    &\leq h^6 C + h^5 C \int_{t_n}^{t_{n+1}} \gamma^2_{\widehat{K}}(\tau)
    \diff{\tau}.
  \end{align*}
\end{proof}

Additionally, we need a bound on the conditional expectation of the local error
$\varrho_{n}$. This result heavily relies on
the randomness of the method.

\begin{lemma}\label{lem:ell2alp}
	Let Assumptions~\ref{ass:f_high_conv} and \ref{ass:rand_var} be fulfilled for $\sigma\in[0,1]$.
	For scheme \eqref{scheme:SDIRK}, it follows that there exist constants
	$C, \widehat{K} >0$ that depend on $M_{\mathrm{apriori}}$, which is stated in
	Lemma~\ref{lem:apriori_bound}, such that
	\begin{align*}
		\big\| \E\big[\varrho_{n+1}\vert \mathcal{F}_{n} \big] \big\|
		&\leq  h^{3+\half} C
		+ h^{2+ 2\min(\sigma,\frac{1}{2})} C \int_{t_n}^{t_{n+1}} \big(L_{\widehat{K}}(\tau)
		+ \gamma_{\widehat{K}}(\tau) \big) \diff{\tau}
	\end{align*}
	is fulfilled for all $n \in \{0,\dots,N-1\}$ and $h = \frac{T}{N}$
	with $h \in (0,h_{\max}]$.
\end{lemma}

\begin{proof}
  For the term $\E\big[ \varrho_{n+1}\vert\mathcal{F}_{n} \big] =
  \E[\Psi(u^n,t_n,h)\vert\mathcal{F}_{n}] - \E[u^{n+1} \vert\mathcal{F}_{n}]$,
  we first recall the expansion for $\Psi$ from \eqref{eq:Taylor_Psi}
	\begin{align*}
		&\E \big[\Psi(u^n,t_n,h)\vert\mathcal{F}_{n}\big]\\
		&= u^n + h f(t_n,u^n) + \frac{h^2}{2} \big(f_t(t_n,u^n) +
		f_u(t_n,u^n) f(t_n,u^n)\big)\\
		&\quad + h\int_0^h (h-\tau)\int_{0}^{1} x^2 F(t_n+ \tau x, \Psi(u^n,t_n,\tau x))\diff{x}\diff{\tau}\\
		&= u^n + h f(t_n,u^n) + \frac{h^2}{2} \big(f_t(t_n,u^n) +
		f_u(t_n,u^n) f(t_n,u^n)\big)\\
		&\quad + h\int_0^h (h-\tau) \int_{\frac{1}{2}}^{1} \big(x^2 F(t_n+ \tau x,
      \Psi(u^n,t_n,\tau x))
    + (1-x)^2 F(t_n+ \tau (1-x), \Psi(u^n,t_n,\tau (1-x)))\big) \diff{x} \diff{\tau},
	\end{align*}
	where $F$ is given in~\eqref{eq:def_F} and we apply that $u^n$ is $\F_n$-measurable, compare Lemma~\ref{lem:exist_measble}.
	For $u^{n+1}$, we apply Lemma~\ref{lem:taylor_approx} to obtain
	\begin{align*}
		\E \big[u^{n+1} \vert\mathcal{F}_{n} \big]
		&= \E \Big[u^n+ hf(t_n,u^n) + \frac{h^2}{2}
		\big(f_t(t_n,u^n)+f_u(t_n,u^n)f(t_n,u^n) \big)\\
		&\quad +\frac{h}{2}\int_0^h (h-\tau) \big(\koness(t_n,\tau,\xi_{n+1},u^n)
		+\ktwoss(t_n,\tau,\xi_{n+1},u^n) \big)\diff{\tau} \Big\vert \mathcal{F}_{n} \Big]\\
		&= u^n+ hf(t_n,u^n) + \frac{h^2}{2}
		\big(f_t(t_n,u^n)+f_u(t_n,u^n))f(t_n,u^n) \big) \\
		&\quad + h \int_{\frac{1}{2}}^{1} \int_0^h (h-\tau) \big(\koness(t_n,\tau,x,u^n)
		+\ktwoss(t_n,\tau,x,u^n) \big)\diff{\tau} \diff{x},
	\end{align*}
	where we use that $\xi_{n+1}$ is independent of $\F_n$ and $u^n$ is
	$\F_n$-measurable, compare Lemma~\ref{lem:exist_measble}. The next step is to apply Lemma~\ref{lem:conc_taylor} and then insert the
	substitutions $t_n + \tau x \to s$, $t_n + \tau (1-x) \to s$ in the first
	integral, $t_n + \tau x \to \tau$ in the second integral and
	$t_n + \tau (1-x) \to \tau$ in the third. For constants $C, \widehat{K}$ that depend on
	$M_{\mathrm{apriori}}$ (compare Lemma~\ref{lem:bounds} for more details), it then follows that
	\begin{align*}
		&\big\|\E\big[\varrho_{n+1} \vert\mathcal{F}_{n}\big]\big\|
    =\big\| \E[\Psi(u^n,t_n,h)\vert\mathcal{F}_{n}] - \E[u^{n+1}\vert\mathcal{F}_{n}] \big\|\\
		&\leq h \int_0^h (h-\tau) \Big\| \int_{\frac{1}{2}}^{1} \big(x^2 F(t_n+ \tau
      x, \Psi(u^n,t_n,\tau x)) + (1-x)^2 F(t_n+ \tau (1-x), \Psi(u^n,t_n,\tau (1-x))) \\
		&\hspace{3cm} - \koness(t_n,\tau,x,u^n)
		-\ktwoss(t_n,\tau,x,u^n) \big) \diff{x} \Big\| \diff{\tau} \\
		&\leq h \int_0^h (h-\tau) \Big( \tau^{\half} C + \tau C \int_{\half}^1 \big(
		\tL_{\tK}(t_n+ \tau x)+\gamma_{\widehat{K}}(t_n+ \tau x) + \gamma_{\widehat{K}}(t_n+ \tau (1-x))\big)\diff{x}\Big)\diff{\tau} \\
		&\hspace{1.5cm} + h C \int_{\half}^1 \int_0^h (h-\tau) \tau^{2\sigma} \big(x^2(L_{\widehat{K}}(t_n+ \tau x)
		+ \gamma_{\widehat{K}}(t_n+ \tau x)) \big) \diff{\tau} \diff{x}\\
		&\hspace{1.5cm} + h C \int_{\half}^1 \int_0^h (h-\tau) \tau^{2\sigma} \big((1-x)^2
		(L_{\widehat{K}}(t_n+ \tau (1-x))+\gamma_{\widehat{K}}(t_n+ \tau (1-x)))\big) \diff{\tau} \diff{x}\\
		&= h \int_0^h (h-\tau) \Big( \tau^{\half} C + C \int_{t_n+ \frac{\tau}{2}}^{t_{n+1}} \big(
		\tL_{\tK}(s)+\gamma_{\widehat{K}}(s) \big)\diff{s}
		+ C \int_{t_n}^{t_n+ \frac{\tau}{2}} \gamma_{\widehat{K}}(s) \diff{s} \Big) \diff{\tau}\\
		&\hspace{1.5cm} + h C \int_{\half}^1 \int_{t_n}^{t_n+ h x} (h-\tau) \tau^{2\sigma} x \big(L_{\widehat{K}}(\tau)
		+ \gamma_{\widehat{K}}(\tau) \big) \diff{\tau} \diff{x}\\
		&\hspace{1.5cm} + h C \int_{\half}^1 \int_{t_n}^{t_n +h(1-x)} (h-\tau) \tau^{2\sigma} (1-x) \big(
		L_{\widehat{K}}(\tau)+\gamma_{\widehat{K}}(\tau) \big) \diff{\tau} \diff{x}\\
		&\leq h^{3+\half} C + h^3 C \int_{t_n}^{t_{n+1}} \big(
		\tL_{\widehat{K}}(s)+\gamma_{\widehat{K}}(s) \big)\diff{s}
		+ h^{2+ 2\sigma} C \int_{t_n}^{t_{n+1}} \big(L_{\widehat{K}}(\tau)
		+ \gamma_{\widehat{K}}(\tau) \big) \diff{\tau}.
	\end{align*}
\end{proof}

Combining our results, we get a convergence rate for our randomized Runge--Kutta
method.
\begin{theorem}\label{thm:conv}
  Let Assumptions~\ref{ass:f_high_conv} and \ref{ass:rand_var} be fulfilled.
  For scheme \eqref{scheme:SDIRK}, it then follows that
  \begin{align*}
    \E\big[\|e_{n+1}\|^2\big]
    \leq h^{4+4 \min(\sigma, \frac{1}{4})} C \exp (C_2 T)
  \end{align*}
  for all $n \in  \{1,\dots,N-1$\} and $h = \frac{T}{N}$ with $h \in (0,h_{\max}]$.
  Note that the constant $C_2$ comes from Lemma~\ref{lem:error_bound} and only depends on $\nu$ and $h_{\max}$. Moreover, the constant $C >0$ depends on
  $M_{\mathrm{exact}}$ and $M_{\mathrm{apriori}}$, which are stated in
  Lemma~\ref{lem:uexist} and Lemma~\ref{lem:apriori_bound}.
\end{theorem}

\begin{proof}
  In this proof, we apply a bootstrap argument to ensure that the error constant
  stays as small as possible and no Lipschitz constants appear in the
  exponential function.
  First we apply Lemma~\ref{lem:error_bound} followed by inserting the bounds
  for $\E [\|\varrho_{n+1}\|^2]$ from Lemma~\ref{lem:E_norm_loc_cond} and for $\E[ \|\E [\varrho_{n+1}\vert 	\mathcal{F}_n ] \| ]$ from Lemma~\ref{lem:ell2alp} to obtain
  \begin{align}
  	\notag
  	&\E\big[\|e_{n+1}\|^2\big] - \E\big[\|e_{n}\|^2\big]\\
  	\notag
  	&\leq h C_2 \E \big[\|e_n\|^2 \big]
  	+ h^2 C \E \big[\|e_n\|^2 \big]
  	+ 2 \E \big[\|\varrho_{n+1}\|^2\big]
  	+ h^{-1} \big( \E\big[\big\|\E\big[\varrho_{n+1}\vert 	\mathcal{F}_n\big]\big\| \big] \big)^2\\
  	&\leq h C_2 \E \big[\|e_n\|^2 \big]
  	+ h^2 C \E \big[\|e_n\|^2 \big]
  	+ h^6 C
  	+ h^{4+ 4 \min(\sigma, \frac{1}{4})} \int_{t_n}^{t_{n+1}} \big(\gamma_{\widehat{K}}^2(\tau) + \tL_{\widehat{K}}^2(\tau)\big) \diff{\tau},
  	\label{eq:proof_conv_bootstrap1}
  \end{align}
  where the constants $C, \widehat{K} >0$ depend on $M_{\mathrm{apriori}}$, which are stated in Lemma~\ref{lem:uexist} and
  Lemma~\ref{lem:apriori_bound}, respectively.
  We can now apply Lemma~\ref{lem:B_conv} for the second summand containing
  $\E[\|e_n\|^2]$ to find
  \begin{equation}\label{eq:proof_conv_bootstrap2}
    \E\big[\|e_{n+1}\|^2\big] - \E\big[\|e_{n}\|^2\big]
    \leq h C_2 \E\big[\|e_n\|^2\big] + h^4 C + h^{4+ 4 \min(\sigma, \frac{1}{4})} \int_{t_n}^{t_{n+1}} \big(\gamma_{\widehat{K}}^2(\tau) + \tL_{\widehat{K}}^2(\tau)\big) \diff{\tau}.
  \end{equation}
  We now sum up this inequality from $k = 0$ to $n$ and apply a telescopic sum
  argument to obtain
  \begin{align*}
    \E\big[\|e_{n+1}\|^2\big] - \|e_0\|^2
    \leq h C_2 \sum_{k=0}^{n} \E\big[\|e_k\|^2\big] + h^3 C+ h^{4+ 4 \min(\sigma, \frac{1}{4})} \int_{0}^{T} \big(\gamma_{\widehat{K}}^2(\tau) + \tL_{\widehat{K}}^2(\tau)\big) \diff{\tau}.
  \end{align*}
  The next step is to apply Gr\"onwall's inequality and the fact that
  $\|e_0\|=0$ to obtain
  \begin{align*}
    \E\big[\|e_{n+1}\|^2\big]
    \leq h^3 C \exp(C_2 T).
  \end{align*}
  Inserting the previous bound in the step from \eqref{eq:proof_conv_bootstrap1}
  to \eqref{eq:proof_conv_bootstrap2} instead of Lemma~\ref{lem:B_conv}, we can
  argue analogously to find
  \begin{align*}
    \E\big[\|e_{n+1}\|^2\big]
    \leq h^4 C \exp(C_2 T).
  \end{align*}
  Applying this bootstrap argument once more, we finally obtain
  \begin{align*}
    \E\big[\|e_{n+1}\|^2\big]
    \leq \Big(h^5 C + h^{4+4 \min(\sigma, \frac{1}{4})} C \int_{0}^{T} \big(\gamma_{\widehat{K}}^2(\tau)
     + \tL_{\widehat{K}}^2(\tau)\big) \diff{\tau}\Big) \exp(C_2 T)
    \leq h^{4+4 \min(\sigma, \frac{1}{4})} C \exp(C_2 T).
  \end{align*}
\end{proof}

\section{Numerical examples}\label{sec:numExp}

To verify our theoretical findings from the previous setting and to show the
advantage of the randomized method, we conduct two numerical experiments.
We compare our scheme with some well-known methods.
The first is the implicit Euler method, sometimes called backward Euler, shown
in Table~\ref{scheme:Imp_Eul}.
This method is a first-order stiffly accurate one-stage method.
The second deterministic method, cf.\ Table~\ref{scheme:det_SDIRK},
is obtained directly from scheme~\eqref{scheme:det_fam_SDIRK}
by setting  $x = \frac{1}{2}(1+\frac{\sqrt{3}}{3})$,
thereby satisfying the third-order conditions.

\begin{table}[ht]
  \begin{minipage}{0.45\linewidth}
    \centering
    \vspace{10pt}
    \begin{tabular}{c|c}
      $1$   & $1$     \\ \hline
      & $1$
    \end{tabular}
    \centering
    \vspace{12pt}
    \caption{Implicit Euler (impl. Euler)}
    \label{scheme:Imp_Eul}
  \end{minipage}
  \hfill
  \begin{minipage}{0.45\linewidth}
    \centering
    \begin{tabular}{c|cc}
      $\frac{1}{2}(1+\frac{\sqrt{3}}{3})$
      & $\frac{1}{2}(1+\frac{\sqrt{3}}{3})$   & $0$   \\
      $1-\frac{1}{2}(1+\frac{\sqrt{3}}{3})$
      & $-\frac{\sqrt{3}}{3}$ & $\frac{1}{2}(1+\frac{\sqrt{3}}{3})$  \\ \hline
      & $\half$ & $\half$
    \end{tabular}
    \centering
    \caption{SDIRK3}
    \label{scheme:det_SDIRK}
  \end{minipage}
\end{table}

The third scheme, cf.\ Table~\ref{scheme:Theta}, is the implicit randomized
$\theta$-method, and the fourth, cf.\ Table~\ref{scheme:ERK}, is an explicit
randomized Runge--Kutta method.
They have been studied in \cite{2025-bochacik_convergence} and \cite{Kruse2017},
and neither is mean-square A-stable.
Note that in contrast to our method, the random variable $\theta$ has a uniform
distribution over $[0,1]$.

\begin{table}[ht]
  \begin{minipage}{0.45\linewidth}
    \centering
    \vspace{10pt}
    \begin{tabular}{c|c}
      $\theta$   & $\theta$     \\ \hline
      & $1$
    \end{tabular}
    \centering
    \vspace{12pt}
    \caption{Randomized $\theta$-method (rand. $\theta$-method)
    \cite{2025-bochacik_convergence}}
    \label{scheme:Theta}
  \end{minipage}
  \hfill
  \begin{minipage}{0.45\linewidth}
    \centering
    \begin{tabular}{c|cc}
      $0$   & $0$  & $0$   \\
      $\theta$ & $\theta$  & $0$ \\ \hline
      & $0$ & $1$
    \end{tabular}
    \centering
    \caption{Randomized explicit RK (rand.~ERK) \cite{Kruse2017}}
    \label{scheme:ERK}
  \end{minipage}
\end{table}

The methods are compared based on their relative error at the final time $t=T$.
For the randomized methods, we report the root-mean-square error averaged over $100$ samples,
given by
\begin{equation*}
  \frac{\big(\E\big[\|e_N\|^2\big]\big)^{\half}}{\|u(T)\|}\approx
  \Big(\frac{1}{100}\sum_{i=1}^{100} \frac{\|u^{(i)}_N-u(T)\|^2}{\|u(T)\|^2}
  \Big)^{\half},
\end{equation*}
where $u^{(i)}_N$ is from the $i$-th sample.

\subsection{Nonlinear stiff equation} \label{sec:nonlinExp}
We assess the method's efficiency and convergence rate for stiff problems
using the following ODE
\begin{equation}\label{eq:AllenCahn}
  \begin{cases}
    u'(t)=u(t)-u(t)^3, \quad t\in[0,0.1],\\
    u(0)=10^2,
  \end{cases}
\end{equation}
which has the exact solution $u(t)=10^2(10^4+(1-10^4)\mathrm{e}^{-2t})^{-\half}$.
Due to the third-order polynomial term, this equation is nonlinear and for large
values of $u$, it portrays a stiff behaviour.
We observe, cf.\ Figure~\ref{fig:Allen},
that the randomized $\theta$-method and the randomized ERK method
only lead to useful approximations for step sizes less than $\sim 10^{-4}$.
This step size restriction is due to the fact that these methods are not A-stable.
While the (implicit) $\theta$-method leads to a bounded solution for larger step
sizes, it still performs worse than the other methods for most step sizes.
Aside from the worse stability properties of these two methods,
both schemes show lower convergence orders than our method.

The implicit Euler method performs well for large step sizes,
but it is outperformed for smaller step sizes by the higher-order methods.
Our randomized method shows good stability properties for all step sizes
and achieves the theoretically proven convergence rate of $2.5$.
In this setting where the exact solution is sufficiently regular,
the third-order SDIRK3 method performs slightly better than our randomized method.

\begin{figure}[ht]
  \includegraphics[width=0.85\textwidth]{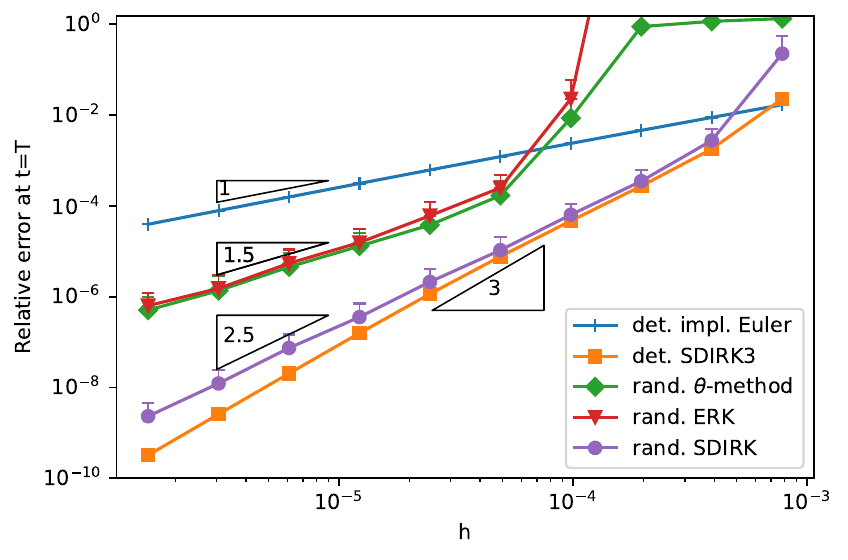}
  \caption{Convergence plot for the nonlinear ODE \eqref{eq:AllenCahn}.
  For the randomized methods, the whiskers indicate the upper bound of the
  one-sided 95\% confidence interval based on the empirical variance.
  The triangles indicate the slopes corresponding to orders 1, 1.5, 2.5 and 3.
  }
  \label{fig:Allen}
\end{figure}

\subsection{Fooling functions} \label{sec:foolingExp}

We observed in the previous example
that the third-order method performs slightly better than our randomized method.
This is to be expected in a setting where the solution is sufficiently regular.
Under lower regularity assumptions, one can often show that randomized schemes
have a higher convergence order compared to deterministic schemes, compare
\cite{Heinrich2008}.
The following example with a highly oscillatory fooling function breaks
a deterministic method while the randomized method is less severely affected.
Consider
\begin{align}\label{eq:PR_fooling_cosp}
  \begin{cases}
    u'(t)= - 10\big(1-g(t)\big)u(t), \quad t\in[0,1],\\
    u(0)=1,
  \end{cases}
\end{align}
with the exact solution $u(t)=\exp (-10(t-\int_0^tg(s)\diff{s}) )$.
First, we choose
\begin{equation}\label{eq:fooling_rhs_cosine}
  g(t)=\frac{1}{10}\cos(2\pi p t),
\end{equation}
where $t \in \R$ and $p=2^{12}$ is the number of oscillations.
To fool the deterministic scheme, we choose the grid $t_n = nh$ for
$h = 2^{-k}$ with $k \in \N$ and $n \in \{1,\dots,2^{k}\}$.
So the equalities
\begin{align*}
  g(t_n + a h)
  = \frac{1}{10} \cos(2\pi p (t_n + a h) )
  = \frac{1}{10} \cos(2\pi p n h + 2\pi p a h )
  = \frac{1}{10} \cos(2\pi p a h )
  = g(a h)
\end{align*}
and
\begin{align*}
  g(a h)
  = \frac{1}{10} \cos(2\pi p a h )
  = \frac{1}{10} \cos(- 2\pi p a h )
  = \frac{1}{10} \cos(2\pi p (1-a) h )
  = g((1-a) h)
\end{align*}
hold for every $a \in (0,1)$ and $h = 2^{-k}$ with $k\leq 12$ since $p h \in \N$.
This shows that, when $k\leq 12$, the evaluations of $g$ for the SDIRK3 method
have all the same value.
Since the SDIRK3 method cannot resolve the oscillations, the numerical error stagnates
regardless of the step size (see Figure~\ref{fig:foolings}, left).
However, the randomized method captures the oscillations
and converges with an observed rate of $0.5$.
Note that for very small step sizes the SDIRK3 method performs well
since \eqref{eq:PR_fooling_cosp} is smooth despite the fast oscillations.
This can already be seen in the last data point of
the left graph of Figure~\ref{fig:foolings}.
Note that we do not observe a convergence order of $2.5$,
which we observed in our theoretical results.
However, the highly oscillating function mimics a non-smooth behaviour.
Thus, for larger step sizes this function acts as a function
that is not even continuous.

In a second test, we choose $g$ to be the Weierstra\ss{} function given by
\begin{equation}\label{eq:fooling_rhs_weierstrass}
  g(t)=\sum_{k=1}^{\infty} 2^{-\frac{k}{2}} \cos(2^{k}\pi t).
\end{equation}
This particular choice of the  Weierstra\ss{} function
is $\alpha$-H\"{o}lder continuous for $\alpha<\half$.
Note that we have to truncate the sum in the numerical experiment.
We only include $30$ summands in our test.
According to \cite{Heinrich2008}, we can expect a convergence rate of $0.5$
for deterministic methods and $1$ for randomized methods,
as observed in Fig.~\ref{fig:foolings} (right).

\begin{figure}[ht]
  \includegraphics[width=0.45\textwidth]{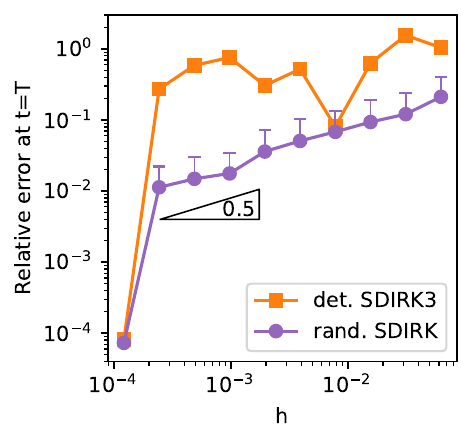}
  \includegraphics[width=0.45\textwidth]{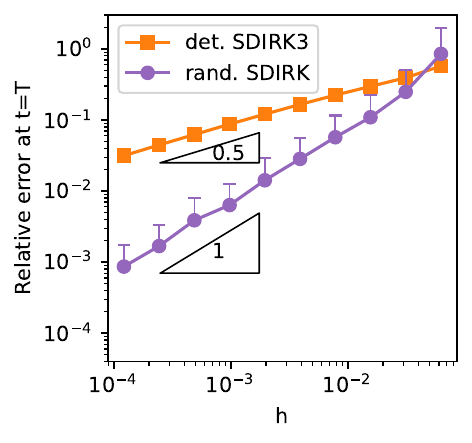}
  \caption{Numerical errors for problem~\eqref{eq:PR_fooling_cosp}
  with the functions~\eqref{eq:fooling_rhs_cosine} (left)
  and~\eqref{eq:fooling_rhs_weierstrass} (right).}
  \label{fig:foolings}
\end{figure}

\bigskip

In summary, the numerical experiments confirm the theoretical analysis. While
the deterministic third-order SDIRK3 scheme is slightly more efficient
for smooth problems, the randomized method is more robust in
low-regularity settings.

\appendix

\section{Auxiliary result} \label{sec:appendix}

\begin{lemma}\label{lem:calc}
  For every $L^1$-integrable function $g \colon \R \rightarrow \R^d$, $h >0$ and
  $t \in \R$, it follows that
  \begin{align*}
    h \int_{0}^{h}  (h -s) \int_{0}^1x^2 g(t + sx) \diff{x} \diff{s}
    = \int_{0}^{h} \frac{(h- s)^2}{2} g(t + s) \diff{s}.
  \end{align*}
\end{lemma}

\begin{proof}
  Using Fubini's Theorem and the substitution $sx \to s$, we obtain
  \begin{align*}
    &h \int_{0}^{h}  (h -s) \int_{0}^1x^2g(t + sx) \diff{x} \diff{s}
    =  \int_{0}^1 \int_{0}^{h} hx^2 (h -s) g(t + sx) \diff{s} \diff{x}\\
    &\quad = \int_{0}^1 \int_{0}^{hx} h x \Big(h -\frac{s}{x}\Big) g(t + s)
    \diff{s} \diff{x}
    = \int_{0}^{h} \int_{\frac{s}{h}}^{1} h (hx -s) \diff{x} \, g(t + s)
    \diff{s} \\
    &\quad = \int_{0}^{h} h \Big[ \frac{h x^2}{2} -sx \Big]_{x = \frac{s}{h}}^{1}
    g(t + s) \diff{s}
    = \int_{0}^{h} \frac{h^2 - 2 s h + s^2}{2} g(t + s) \diff{s}\\
    &\quad = \int_{0}^{h} \frac{(h- s)^2}{2} g(t + s) \diff{s}.
  \end{align*}
\end{proof}

\bibliographystyle{amsplain}
\bibliography{lit.bib}

\providecommand{\bysame}{\leavevmode\hbox to3em{\hrulefill}\thinspace}
\providecommand{\MR}{\relax\ifhmode\unskip\space\fi MR }
\providecommand{\MRhref}[2]{%
  \href{http://www.ams.org/mathscinet-getitem?mr=#1}{#2}
}
\providecommand{\href}[2]{#2}
\begin{thebibliography}{10}

\bibitem{Bochacik.2023B}
T.~Bochacik, \emph{A note on the probabilistic stability of randomized {T}aylor
  schemes}, Electron. Trans. Numer. Anal. \textbf{58} (2023), 101--114.

\bibitem{Bochacik.2023}
\bysame, \emph{On the properties of the exceptional set for the randomized
  {E}uler and {R}unge--{K}utta schemes}, Adv. Comput. Math. \textbf{49} (2023),
  no.~2, Paper No. 14, 16.

\bibitem{Bochacik2021}
T.~Bochacik, M.~Goćwin, P.~M. Morkisz, and P.~Przybyłowicz, \emph{Randomized
  {Runge}--{Kutta} method — {Stability} and convergence under inexact
  information}, Journal of Complexity \textbf{65} (2021), Paper No. 101554, 21.

\bibitem{2025-bochacik_convergence}
T.~Bochacik, P.~Przybyłowicz, and Ł. Stępień, \emph{Convergence and
  stability of randomized implicit two-stage {Runge}--{Kutta} schemes}, BIT
  Numerical Mathematics \textbf{65} (2025), no.~1, 23.

\bibitem{1975-Butcher}
J.~C. Butcher, \emph{A stability property of implicit {Runge}--{Kutta}
  methods}, BIT \textbf{15} (1975), no.~4, 358--361.

\bibitem{Butcher1987}
\bysame, \emph{The numerical analysis of ordinary differential equations}, A
  Wiley-Interscience Publication, John Wiley \& Sons, Ltd., Chichester, 1987,
  Runge--Kutta and general linear methods.

\bibitem{Clark1987}
D.~S. Clark, \emph{Short proof of a discrete {G}ronwall inequality}, Discrete
  Appl. Math. \textbf{16} (1987), no.~3, 279--281.

\bibitem{Coulibaly1999}
I.~Coulibaly and C.~L\'ecot, \emph{A quasi-randomized {R}unge--{K}utta method},
  Math. Comp. \textbf{68} (1999), no.~226, 651--659.

\bibitem{Dahlquist1963}
G.~G. Dahlquist, \emph{A special stability problem for linear multistep
  methods}, Nordisk Tidskr. Informationsbehandling (BIT) \textbf{3} (1963),
  27--43.

\bibitem{Daun2011}
T.~Daun, \emph{On the randomized solution of initial value problems}, J.
  Complexity \textbf{27} (2011), no.~3-4, 300--311.

\bibitem{Deimling1977}
K.~Deimling, \emph{Ordinary differential equations in {B}anach spaces}, Lecture
  Notes in Mathematics, vol. Vol. 596, Springer-Verlag, Berlin-New York, 1977.

\bibitem{Deimling.1985}
\bysame, \emph{Nonlinear functional analysis}, Springer-Verlag, Berlin, 1985.
  \MR{787404}

\bibitem{Dodson}
M.~A. Dodson, \emph{{Extended Runge--Kutta Monte Carlo Methods}}, Phd thesis,
  Lehigh University, Example City, CA, 1994, Available at \url{
  https://preserve.lehigh.edu/lehigh-scholarship/graduate-publications-theses-dissertations/theses-dissertations/extended-runge
  }.

\bibitem{Eisenmann2019}
M.~Eisenmann, M.~Kov\'acs, R.~Kruse, and S.~Larsson, \emph{On a randomized
  backward {E}uler method for nonlinear evolution equations with time-irregular
  coefficients}, Found. Comput. Math. \textbf{19} (2019), no.~6, 1387--1430.

\bibitem{Emmrich.2004}
E.~Emmrich, \emph{Gew{\"o}hnliche und {O}perator-{D}ifferentialgleichungen:
  {E}ine integrierte {E}inf{\"u}h\-rung in {R}andwertprobleme und
  {E}volutionsgleichungen f{\"u}r {S}tudierende}, Vieweg-Studium : Mathematik,
  Vieweg+Teubner Verlag, 2004.

\bibitem{Evans2010}
L.~C. Evans, \emph{{Partial Differential Equations, Second edition}}, American
  Mathematical Society, 2010.

\bibitem{haber1966}
S.~Haber, \emph{A modified {M}onte-{C}arlo quadrature}, Math. Comp. \textbf{20}
  (1966), 361--368.

\bibitem{haber1967}
\bysame, \emph{A modified {M}onte-{C}arlo quadrature. {II}}, Math. Comp.
  \textbf{21} (1967), 388--397.

\bibitem{Hairer1996}
E.~Hairer and G.~Wanner, \emph{{Solving Ordinary Differential Equations II}},
  Springer Berlin Heidelberg, 1996.

\bibitem{Hale1980}
J.~K. Hale, \emph{Ordinary differential equations}, second ed., Robert E.
  Krieger Publishing Co., Inc., Huntington, NY, 1980.

\bibitem{Heinrich2008}
S.~Heinrich and B.~Milla, \emph{The randomized complexity of initial value
  problems}, J. Complexity \textbf{24} (2008), no.~2, 77--88.

\bibitem{higham2000}
D.~J. Higham, \emph{Mean-square and asymptotic stability of the stochastic
  theta method}, SIAM J. Numer. Anal. \textbf{38} (2000), no.~3, 753--769
  (electronic).

\bibitem{Hofmanova.2020}
M.~Hofmanov\'a, M.~Kn\"oller, and K.~Schratz, \emph{Randomized exponential
  integrators for modulated nonlinear {S}chr\" odinger equations}, IMA J.
  Numer. Anal. \textbf{40} (2020), no.~4, 2143--2162.

\bibitem{HuEtAll.2024}
Y.~Hu, F.~Cheng, B.~Hu, and S.~Sarwar, \emph{Mathematical analysis of a
  randomized method for fractional {C}arath\'eodory type equation with
  time-irregular coefficients}, ZAMM Z. Angew. Math. Mech. \textbf{104} (2024),
  no.~5, Paper No. e202300264, 20.

\bibitem{Jentzen2009}
A.~Jentzen and A.~Neuenkirch, \emph{A random {E}uler scheme for
  {C}arath\'eodory differential equations}, J. Comput. Appl. Math. \textbf{224}
  (2009), no.~1, 346--359.

\bibitem{Kacewicz2004}
B.~Kacewicz, \emph{Randomized and quantum algorithms yield a speed-up for
  initial-value problems}, J. Complexity \textbf{20} (2004), no.~6, 821--834.

\bibitem{Kainhofer2003}
R.~Kainhofer, \emph{Q{MC} methods for the solution of delay differential
  equations}, J. Comput. Appl. Math. \textbf{155} (2003), no.~2, 239--252.

\bibitem{Kruse2017}
R.~Kruse and Y.~Wu, \emph{Error analysis of randomized {R}unge--{K}utta methods
  for differential equations with time-irregular coefficients}, Comput. Methods
  Appl. Math. \textbf{17} (2017), no.~3, 479--498.

\bibitem{KruseWu.2019}
\bysame, \emph{A randomized and fully discrete {G}alerkin finite element method
  for semilinear stochastic evolution equations}, Math. Comp. \textbf{88}
  (2019), no.~320, 2793--2825.

\bibitem{Lecot2001}
C.~Lécot, \emph{Quasi-randomized numerical methods for systems with
  coefficients of bounded variation}, Mathematics and Computers in Simulation
  \textbf{55} (2001), no.~1, 113--121, The Second IMACS Seminar on Monte Carlo
  Methods.

\bibitem{saito1996}
Y.~Saito and T.~Mitsui, \emph{Stability analysis of numerical schemes for
  stochastic differential equations}, SIAM J. Numer. Anal. \textbf{33} (1996),
  no.~6, 2254--2267.

\bibitem{stengle1990}
G.~Stengle, \emph{Numerical methods for systems with measurable coefficients},
  Appl. Math. Lett. \textbf{3} (1990), no.~4, 25--29.

\bibitem{stengle1995}
\bysame, \emph{Error analysis of a randomized numerical method}, Numer. Math.
  \textbf{70} (1995), no.~1, 119--128.

\bibitem{Strehmel2012}
K.~Strehmel, R.~Weiner, and H.~Podhaisky, \emph{{Numerik gew\"{o}hnlicher
  Differentialgleichungen}}, Vieweg+Teubner Verlag Wiesbaden, 2012.

\bibitem{Wu2022}
Y.~Wu, \emph{A randomized trapezoidal quadrature}, International Journal of
  Computer Mathematics \textbf{99} (2022), no.~4, 680--692.

\end{thebibliography}
\end{document}